\documentclass{amsart}
\usepackage{graphicx,epstopdf}

\newtheorem{thm}{Theorem}[section]
\newtheorem{cor}[thm]{Corollary}
\newtheorem{lem}[thm]{Lemma}
\newtheorem{prop}[thm]{Proposition}
\newtheorem{ex}[thm]{Example}
\theoremstyle{definition}
\newtheorem{defn}[thm]{Definition}
\theoremstyle{remark}
\newtheorem{rem}[thm]{Remark}
\numberwithin{equation}{section}
\begin{document}

\title[]{on partial and generic uniqueness of block term tensor decompositions}
\author{Ming Yang}
\address{Mathematics Department of Texas A\&M University}
\email{mingyang@math.tamu.edu}
\thanks{}%
\subjclass{classical algebraic geometry, multilinear algebra}%

%\date{2010}%
%\dedicatory{}%
% ----------------------------------------------------------------
\begin{abstract}
We present several conditions for generic uniqueness of tensor
decompositions of multilinear
rank $(1,\ L_{1},\ L_{1}),\cdots,(1,\ L_{R},\ L_{R})$ terms.
In geometric language, we
prove that the joins of relevant subspace
varieties are not tangentially weakly defective. We also give conditions for partial uniqueness of
block term tensor decompositions by proving that the joins
of relevant subspace varieties are not defective.
\end{abstract}
\maketitle{}
% ----------------------------------------------------------------
\section{Introduction and Main results}
Recently, De Lathauwer (\cite{1LL,Lat 1,Lat 2,Lat 3}) introduced
the concept of block term tensor decompositions, because it is natural for certain
source separation problems in signal processing, and often has
better uniqueness properties than decompositions by tensor rank. We refer the reader to \cite{Lat 4} for the applications of block
term decomposition in blind source separation. Therefore, the study of the uniqueness property of this kind of tensor decompositions is of interest. The results of this paper, mainly concerning the uniqueness property, extend the range of applicability of block term tensor decompositions. To present these results, we first recall the definition of a block term decomposition of a tensor. Throughout this paper, for basic definitions, notation and results, we follow \cite{La}, which is addressed to both the numerical and the algebraic geometrical research communities.

Let $A_{j}$, $1\leq j\leq n$, be finite dimensional complex vector spaces. The elements of $A_{1}\otimes\ldots\otimes A_{n}$ are called $n$-tensors. When no confusion can occur, they are simply called tensors. For tensors there are several different notions of rank that we review below.

For $\beta_{j}\in A^{*}_{j}$, $1\leq j\leq n$, where $A^{*}_{j}$ is the dual space of $A_{j}$, let an element $\beta_{1}\otimes\ldots\otimes \beta_{n}$ denote the unique element in $A^{*}_{1}\otimes\ldots\otimes A^{*}_{n}$ determined by
the condition
\begin{align}\label{rank 1}
(\beta_{1}\otimes\ldots\otimes \beta_{n})\vdash
(v_{1},\cdots,v_{n}):=\beta_{1}(v_{1})\cdots \beta_{n}(v_{n}),\ v_{j}\in A_{j}.
\end{align}
An element in $A^{*}_{1}\otimes\ldots\otimes A^{*}_{n}$ is said to have \textit{rank}
one if it can be written as in \eqref{rank 1}. Using the obvious reflexity $(A^{*}_{1}\otimes\ldots\otimes A^{*}_{n})^{*}=A_{1}\otimes\ldots\otimes A_{n}$, a rank one tensor in $A_{1}\otimes\ldots\otimes A_{n}$ is defined similarly.
\begin{defn}
Define the \textit{rank} of a tensor $T\in A_{1}\otimes\ldots\otimes A_{n}$, denoted
$\mathbf{R}(T)$, to be the minimal number $r$ such that
\begin{align}
T=\xi_{1}+\cdots+\xi_{r},
\end{align}
with each $\xi_{j}$ of rank one.
\end{defn}
\begin{defn}\label{tensor hom}
When studying tensors in $A_{1}\otimes\cdots\otimes A_{n}$, it is convenient to introduce the notation $\hat{A}_{\hat{\jmath}}:=
A_{1}\otimes\cdots A_{j-1} \otimes A_{j+1}\otimes A_{n}$. Also, given $T \in A_{1}\otimes\cdots\otimes A_{n}$, it canonically defines a linear map $A^{*}_{j}\rightarrow \hat{A}_{\hat{\jmath}}$ for all $j\in \{1,\ldots,n\}$. The image of this map will be denoted by $T(A^{*}_{j})\subset \hat{A}_{\hat{\jmath}}$
and the image of the transpose will be denoted by $T^{t}(\hat{A}^{*}_{\hat{\jmath}})\subset A_{j}$.
\end{defn}
\begin{defn}
The \textit{multilinear rank} of $T \in A_{1}\otimes\ldots\otimes A_{n}$ is the $n$-tuple of natural numbers
\begin{align}
\mathbf{R}_{multlin}(T):=(\textrm{dim}\ T(A^{*}_{1}),\cdots , \textrm{dim}\ T(A^{*}_{j}),\cdots , \textrm{dim}\ T(A^{*}_{n})).
\end{align}
The number $\textrm{dim}\ T(A^{*}_{j})$ is called\textit{ the mode $j$ rank} of $T$.
\end{defn}
\begin{rem}
Observe that for a matrix (i.e., the case $n=2$), the rank and the mode-$1$ and mode-$2$ ranks are all equal.
\end{rem}
\begin{rem}(see Page $34$ Exercise $2.4.2.6$ in \cite{La})
If $T \in A_{1}\otimes\ldots\otimes A_{n}$, then the multilinear rank $(b_{1},\cdots, b_{n})$
of $T$ satisfies $b_{i} \leq \textrm{min}\ (a_{i}, \prod_{j\neq i} a_{j})$ and equality holds for
general tensors. 
\end{rem}
\begin{rem}
For more information of multilinear rank, we recommend the reader to see \cite{Car}.
\end{rem}

\begin{defn}\label{BTD1LL}(see Definition $2.1$ in \cite{Lat 4})
A \textit{block term tensor decomposition} of a tensor $Y\in \mathbb{C}^{I}\otimes
\mathbb{C}^{J}\otimes \mathbb{C}^{K}$ in a sum of multilinear
rank $(1,\ L_{1},\ L_{1}),\cdots,(1,\ L_{R},\ L_{R})$
terms, is a
decomposition of $Y$ of the form
\begin{align}\label{BTD1LL}
Y=\sum_{r=1}^{R}a_{r}\otimes X_{r},
\end{align}
in which $a_{r}\in \mathbb{C}^{I}$, and $X_{r}\in
\mathbb{C}^{J}\otimes \mathbb{C}^{K}$ is of rank $L_{r}$. (Each
term consists of the outer product of a vector and a rank-$L_{r}$
matrix.)
\end{defn}
In \eqref{BTD1LL}, one can permute the $r$-th and $r'$-th term
when $L_{r} = L_{r'}$ . Also one can scale $X_{r}$, provided
$a_{r}$ is counter scaled as well. The decomposition is said to be
\textit{essentially unique} when it is subject only to these
trivial identifications. The decomposition is said to be
\textit{partially unique} when it has finite number of representations.

The main results in this paper are the following.
\begin{thm}\label{main results}
Assume that $L_{1}\leq L_{2}\leq \ldots \leq
L_{R},\ K\geq J> L_{R}$. If
\begin{align}\label{dim ambiant}
\sum_{1\leq r\leq R}(J\cdot L_{r}+L_{r}\cdot
(K-L_{r})+I-1)<IJK,
\end{align}
then for \emph{general} tensors among those admitting block term tensor decomposition as in \eqref{BTD1LL}, the block term tensor decomposition

$(\imath)$ is partially unique under the condition
\begin{align*}
(A)\ {{J}\choose{L_{R}}}\geq R,\ I\geq 2,
\end{align*}

$(\imath\imath)$ has infinitely many expressions under the condition
\begin{align*}
(B)\ IJK<L^{2}_{1}+\cdots+L^{2}_{R},
\end{align*}

$(\imath\imath\imath)$ is essentially unique under the following conditions:
\begin{align*}
&(C)\ I\geq2,\ J,K\geq\sum^{R}_{r=1}L_{r},\\&
(D)\ R=2,\ I\geq 2,\\&
(E)\ I\geq R,\ K\geq
\sum^{R}_{r=1}L_{r},\ J\geq 2L_{R}, {{J}\choose{L_{R}}}\geq
R.
\end{align*}
\end{thm}

\begin{rem}
Here the meaning of \emph{general}  is that, the set of tensors which do
not have the respective uniqueness property is included in a proper subvariety (see Remark \ref{norm form}). Relation $(D)$ in $(\imath\imath\imath)$ also appears as Theorem $2.2$ in \cite{Lat 4} in a different context.  The other relations are new. Our proof is based on algebraic geometric methods presented in the next section. Also, the hypothesis \eqref{dim ambiant} is actually not restrictive (see the beginning of Section $4$). And if $\sum_{1\leq r\leq R}(J\cdot L_{r}+L_{r}\cdot (K-L_{r})+I-1)>IJK$, then \eqref{BTD1LL} has infinitely many expressions.
\end{rem}
Using similar methods, we also establish results valid for:

$(\iota)$ any tensors and any multilinear rank (see Proposition \ref{def of many join}),

$(\iota\iota)$ any $3$-tensors of any multilinear rank (see Proposition \ref{defect of 2 join}).

\subsection{Outline of the paper} In Section $2$ we develop the basic notions on subspace
varieties, defectivity and tangential weak defectivity. At the end of this section we prove Propositions \ref{def of many join} and \ref{defect of 2 join}. Section $3$ contains some examples relevant to our study. In Section $4$ we prove our main Theorem \ref{main results}. Finally, in Section $5$ we give a new proof of  De Lathauwer's criterion of uniqueness of block term tensor decompositions (Theorem $2.3$ in \cite{Lat 4}).
\section{Basic Algebraic Geometric Methods}
\subsection{Notations}

As in \cite{La}, for a finite dimensional complex vector space $V$, $\mathbb{P}V$ denotes the projective space associated to $V$, $\pi$ denotes the projection of $V \backslash \{0\}$ onto  $\mathbb{P}V$; for a variety $X \subset \mathbb{P}V$, $\hat{X}\subset V$ denotes its inverse image under the
projection $\pi$, which is the (affine) cone over $X$ in $V$, and for $x\in X$, $[x]$ denotes $\pi(x)$.

Let $S$ be a subset of $\mathbb{P}V$, then the span $\langle S \rangle$ is by definition the range of $\pi$ on the usual vector span of $\hat{S}$ in $V$.
The Zariski closure of $S$ in $\mathbb{P}V$ will be denoted by $\bar{S}$.

When we will need to specify the the elements of $S$ and its linear span, we use the notation $\{s_{1},s_{2}, \cdots\}$ and $\langle s_{1},s_{2}, \cdots \rangle$ respectively.

For $x\in \hat{X}$, $\hat{T}_{[x]}X:=\hat{T}_{x}\hat{X}$ is the affine tangent space to $X$ at $[x]$.

For a vector space $V$, its dual space is denoted $V^{*}$. If $A
\subset V$ is a subspace, $A^{\bot}\subset V^{*}$ is its
annihilator, namely the space of $f\in V^{*}$, satisfying $f(a)=0,\
\textrm{for}\ \textrm{all}\ a \in A$.

$\hat{T}^{\bot}_{x}X:=(\hat{T}_{x}X)^{\bot}$ is the affine conormal space of $X$ at $x$.

\subsection{Terracini's lemma and its Corollaries}
Terracini \cite{CC, CG} introduced an algebraic
geometric criterion of uniqueness of tensor decomposition. We will use a Corollary of Terracini's lemma, which appears as Proposition $2.4$ in \cite{CC} for secant varieties.

\begin{defn}\label{def sub}(See Definition $1$ in \cite{LaWe})
\textit{Subspace varieties}, denoted
$Sub_{k_{1},\ldots,k_{n}}(A_{1}\otimes\ldots\otimes A_{n})\in
\mathbb{P}(A_{1}\otimes\ldots\otimes A_{n})$ are defined as
\begin{align*}
&Sub_{k_{1},\ldots,k_{n}}(A_{1}\otimes\ldots\otimes A_{n})\\
&:= \overline{\{[T]\in \mathbb{P}(A_{1}\otimes\ldots\otimes A_{n})\mid
\forall i\ \exists A'_{i}\subset A_{i},\ \textrm{dim}\
A'_{i}=k_{i},\ T\in A'_{1}\otimes\ldots\otimes
A'_{n}\}}\\
&:=\{[T]\in \mathbb{P}(A_{1}\otimes\ldots\otimes A_{n})\mid
\forall i\ \dim T(A^{*}_{i})\leq k_{i}\}.
\end{align*}
\end{defn}

\begin{rem}
The multilinear rank of a tensor $[T]\in
\mathbb{P}(A_{1}\otimes...\otimes A_{n})$ is the minimum
$(k_{1},\ldots,k_{n})$ such that $[T]\in
Sub_{k_{1},\ldots,k_{n}}(A_{1}\otimes...\otimes A_{n})$. And the general elements (if they exist) in $Sub_{k_{1},\ldots,k_{n}}(A_{1}\otimes...\otimes A_{n})$ are of multilinear rank-$(k_{1},\ldots,k_{n})$.
\end{rem}

\begin{defn}
If $X_{i},\ i=1,\ldots,k$, $k\leq n$ are projective algebraic
varieties of $\mathbb{P}^{n}=\mathbb{P}V, V=\mathbb{C}^{n+1}$, then the \textit{join} of
$X_{1},\ldots,X_{k}$ is
\begin{align*}
\textbf{J}(X_{1},\ldots,X_{k}):=
\overline{\cup\{\langle [P_{1}],\ldots,[P_{k}]\rangle| P_{i}\in \hat{X}_{i},\ 1\leq i\leq k\}},
\end{align*}
where $P_{i}$, $i=1,\ldots, k$, are linearly independent vectors in $V$.
If $X_{1}=\cdots=X_{k}=X$, then we write
$\textbf{J}(X_{1},...,X_{k})=\sigma_{k}(X)$ and we call this the
$k$-th secant variety to X.
\end{defn}

In our considerations, the following fact will play an essential role.
\begin{rem}\label{norm form}
There was a normal form
for a point $[p]$ of $\sigma_{L}(\mathbb{P}B\times \mathbb{P}C)$ (see Proposition 5.3.0.5 in \cite{La} and also Chapter $11$ in \cite{La}), which is of the form
\begin{align*}
p=b_{1}\otimes c_{1}+\cdots+b_{L}\otimes c_{L}.
\end{align*}
and we may assume that all the $b_{i},1\leq i\leq L$ are linearly independent in $B$ as
well as all the $c_{i},1\leq i\leq L$ (otherwise one would have $[p]\in \sigma_{L-1}(\mathbb{P}B\times \mathbb{P}C)$),
then a general element $[\varphi]\in Sub_{1,L,L}(\mathbb{P}A\otimes \mathbb{P}B\times \mathbb{P}C)$ is of the form
\begin{align*}
\varphi=a_{1}\otimes (b_{1}\otimes c_{1}+\cdots+b_{L}\otimes c_{L}),
\end{align*}
where $a_{1}$ is a nonzero vector in $A$.
\end{rem}
Throughout up to the end of this section, $X_{1},\ldots,X_{k}$ will be as in the above definition.
\begin{thm}\label{ter lem} (Lemma $5.3.0.2$ in \cite{La}, \emph{Terracini's Lemma} of \cite{Ter} in modern language)

Let $P_{i}\in \hat{X}_{i}$ be a general point of $\hat{X}_{i}$, for each $i=1,\ldots, k$, then for
$[P]:=[P_{1}+ \ldots +P_{k}]$,
\begin{align}\label{terracini}
\hat{T}_{[P]}\mathbf{J}(X_{1},\ldots,X_{k})=
\hat{T}_{[P_{1}]}X_{1}+ \cdots +\hat{T}_{[P_{k}]}X_{k}.
\end{align}
\end{thm}
This result reduces the determination of the dimension of the join $\textbf{J}(Sub_{1,L_{1},L_{1}}(\mathbb{C}^{I}\otimes
\mathbb{C}^{J}\otimes
\mathbb{C}^{K}),\cdots,Sub_{1,L_{R},L_{R}}(\mathbb{C}^{I}\otimes
\mathbb{C}^{J}\otimes \mathbb{C}^{K}))$ to be the
calculation of the dimension of the sum of the tangent spaces at general points of our varieties
$Sub_{1,L_{r},L_{r}}(\mathbb{C}^{I}\otimes
\mathbb{C}^{J}\otimes \mathbb{C}^{K})),1\leq r\leq R$.
\begin{defn}
The \textit{expected dimension} of $\textbf{J}(X_{1},\ldots,X_{k})$ is $\min
\{d_{1}+\cdots+d_{k}+k-1,\ n\}$, where $d_{i}=\dim X_{i}$, $1\leq i\leq k$. The \textit{defect} of $\textbf{J}(X_{1},\ldots,X_{k})$ is
\begin{align*}
\delta(\textbf{J}(X_{1},\ldots,X_{k}))=d_{1}+\cdots+d_{k}+k-1-
\textrm{dim}\ \textbf{J}(X_{1},\ldots,X_{k}).
\end{align*}
When the defect of
$\textbf{J}(X_{1},\ldots,X_{k})$ is positive, we say that
$\textbf{J}(X_{1},\ldots,X_{k})$ is \textit{defective}.
\end{defn}
\begin{rem}
By the upper semicontinuity of dimension of tangent space (see Exercise $II.3.22$ of \cite{Hart}), if for one particular set of general points $\{P_{1},\cdots,P_{k}\}$, $\hat{T}_{[P_{1}]}X_{1}+ \cdots +\hat{T}_{[P_{k}]}X_{k}$
has the expected dimension, then  $\textbf{J}(X_{1},\ldots,X_{k})$ is
not defective.
\end{rem}
\begin{cor}\label{partial unique}
If $\textbf{J}(X_{1},\ldots,X_{k})$ is not defective, then general points $[\varphi]$ on $\textbf{J}(X_{1},\ldots,X_{k})$ have a finite
number of decompositions
\begin{align*}
\varphi=P_{1}+\cdots+P_{k},
\end{align*}
with $[P_{i}]\in X_{i}$, $1\leq i\leq k$. Moreover, if $\textbf{J}(X_{1},\ldots,X_{k})$ is defective, then all $[\varphi]\in \textbf{J}(X_{1},\ldots,X_{k})$ have infinitely-many decompositions.
\end{cor}
\begin{proof}
See Corollary $5.3.0.4$ in \cite{La}.
\end{proof}
\begin{defn}(see Definition $2.6$ in \cite{CG})
Let $P_{i}\in \hat{X}_{i}$ be a general point of $\hat{X}_{i}$, $1\leq i\leq k$. When for any $j\in \{1,\ldots,k\}$ and $Q_{j}\in \hat{X}_{j}$, $\hat{T}_{[P_{1}]}X_{1}+\cdots+\hat{T}_{[P_{k}]}X_{k}$
contains $\hat{T}_{[Q_{j}]}X_{j}$ only if $[Q_{j}]\in \{[P_{1}],\ldots,[P_{k}]\}$, we say
$\textbf{J}(X_{1},\ldots,X_{k})$ is \textit{not tangentially weakly
defective}. Otherwise, we say that $\textbf{J}(X_{1},\ldots,X_{k})$ is
\textit{tangentially weakly defective}.
\end{defn}
\begin{rem}
By semicontinuity (see Theorem  $III.12.8$ of \cite{Hart}), if for one particular set of general points $\{P_{1},\ldots,P_{k}\}$, $\hat{T}_{[P_{1}]}X_{1}+\cdots+\hat{T}_{[P_{k}]}X_{k}$
contains $\hat{T}_{[Q_{j}]}X_{j}$ only if $[Q_{j}]\in \{[P_{1}],\ldots,[P_{k}]\}$, then  $\textbf{J}(X_{1},\ldots,X_{k})$ is
tangentially weakly defective.
\end{rem}

\begin{rem}
Notice that if $\textbf{J}(X_{1},\ldots,X_{k})$ is defective, then it
is tangentially weakly defective, but the converse is not true (see Example \ref{222 444}).
\end{rem}

\begin{rem}\label{nor terr}
The equality in Theorem \ref{ter lem} is equivalent to
\begin{align*}
\hat{T}^{\bot}_{[P]}\textbf{J}(X_{1},\ldots,X_{k})=\bigcap_{1\leq
i\leq k}\hat{T}^{\bot}_{[P_{i}]}X_{i}.
\end{align*}
Moreover $\hat{T}_{[P_{1}]}X_{1}+\cdots+\hat{T}_{[P_{k}]}X_{k}\supset
\hat{T}_{[Q_{j}]}X_{j}$ is equivalent to
\begin{align*}
\bigcap_{1\leq i\leq k}\hat{T}^{\bot}_{[P_{i}]}X_{i}\subset
\hat{T}^{\bot}_{[Q_{j}]}X_{j};
\end{align*}
so if $\textbf{J}(X_{1},\ldots,X_{k})$ is  tangentially weakly
defective, then every hyperplane in
$\bigcap_{1\leq i\leq
k}\hat{T}^{\bot}_{[P_{i}]}X_{i}$ is also tangent at $[Q_{j}]\in
X_{j}$.
\end{rem}
We will need the following generalization of of Proposition $2.4$ in \cite{CC}.
\begin{cor}\label{twd unique}
If $\textbf{J}(X_{1},\ldots,X_{k})$ is not tangentially weakly
defective, then for general $[\varphi]\in
\textbf{J}(X_{1},\ldots,X_{k})$, the decomposition
\begin{align*}
[\varphi]=[P_{1}+\cdots+P_{k}],
\end{align*}
with $[P_{i}]\in X_{i}$, $1\leq i\leq k$, is essentially unique.
\end{cor}

\begin{proof}
The proof proceeds like the proof of Proposition $2.4$ in \cite{CC}. Assume the contrary, let us take a general point $[\varphi]\in
\textbf{J}(X_{1},\ldots,X_{k})$ and
\begin{align*}
[\varphi]=[P_{1}+\cdots+P_{k}]=[Q_{1}+\cdots+Q_{k}],
\end{align*}
with $[Q_{i}]\in X_{i}$, $1\leq i\leq k$, and at least one of them, say $[Q_{j}]\in X_{j}$,
not belong to $\{[P_{1}],\ \ldots,\ [P_{k}]\}$. Then by Lemma \ref{terracini}, $\hat{T}_{[P_{1}]}X_{1}+\cdots +\hat{T}_{[P_{k}]}X_{k}$
also contains the tangent space of $X_{j}$ at $[Q_{j}]$. Hence we get a contradiction, and Corollary  \ref{twd unique} follows.
\end{proof}

\begin{defn}(see Page 6 in \cite{CG})
The secant variety $\sigma_{k}(X)$ is \textit{weakly defective} if the general hyperplane
which is tangent to $X$ at some $k$ general points
$[P_{1}],\ldots,[P_{k}]$, is also tangent at some other point $[Q]\neq
[P_{1}],\ldots,[P_{k}]$. Here general means in an open subset of the set of hyperplanes which
are tangent to $X$ at $k$ general family of points $[P_{1}],\cdots,[P_{k}]$.
\end{defn}
\begin{rem}
By Theorem \ref{ter lem}, if $\textbf{J}(X_{1},\ldots,X_{k})$ is
defective then $\textbf{J}(X_{1},\ldots,X_{k})$ is weakly defective,
but the converse is not necessarily true (see Example \ref{weak by not tan-weak}). To sum up, we have the
following relationship
\begin{align*}
\{\textrm{defectivity}\}\subset \{\textrm{tangentially\
weak-defectivity}\}\subset \{\textrm{weak-defectivity}\}
\end{align*}
\end{rem}
\subsection{Infinitesimal study of subspace varieties}
For the benefit of the reader, we recall the following standard notations (see any textbook of algebraic geometry, for example \cite{Weyman}).

Let $V= \mathbb{C}^{n}$, $G(m,V)$ denote the
Grassmannian of $m$-planes through the origin in $V$. It is a
smooth compact algebraic variety of dimension $m(n-m)$.

The {\em{trivial bundle}} $G(m,V)\times V,\ V\cong \mathbb{C}^{n}$
over $G(m,V)$ contains the {\em{universal subbundle}}
$\mathcal{S}$ of rank $m$ that consists of the pairs $(E,v)$ with
$v\in E$. The {\em{quotient bundle}} $\mathcal{Q}$ over $G(m,V)$
of rank $n-m$ whose fiber over $E$ is canonically isomorphic to
$V/E$. These fit into the exact sequence
\begin{align*}
0\longrightarrow \mathcal{S}\longrightarrow G(m,V)\times
V\longrightarrow \mathcal{Q}\longrightarrow 0.
\end{align*}
The following Lemma is well known:
\begin{lem}\label{grassman lemma}
There is a canonical bundle isomorphism
\begin{align*}
TG(m,V)=\mathcal{Q}\otimes \mathcal{S}^{*}
\end{align*}
corresponding to the canonical isomorphism
\begin{align*}\label{tangentgrass}
T_{E}G(m,V)\cong V/E\otimes E^{*}.
\end{align*}
\end{lem}
\begin{lem}\label{tangentspace}
Let $A_{1},\ldots,A_{n}$ and $A'_{1},\ldots,A'_{n}$ be as in Definition \ref{def sub}. For general $\varphi\in A'_{1}\otimes\ldots\otimes A'_{n}$, we have

\begin{align}
\hat{T}_{\varphi}(\widehat{Sub}_{k_{1},\ldots,k_{n}}(A_{1}\otimes\cdots\otimes
A_{n}))=(A'_{1}\otimes\cdots\otimes A'_{n})+\sum_{1\leq i\leq n}
(A_{i}\otimes \varphi(A'^{*}_{i})),
\end{align}
where $\varphi(A'^{*}_{i}),1\leq i\leq n$ is defined as in Definition \ref{tensor hom}.
\end{lem}

\begin{proof}
First, we recall the following Kempf-Weyman desingularization for $\widehat{Sub}_{k_{1},\ldots,k_{n}}(A_{1}\otimes\ldots\otimes
A_{n})$ as in section $7.4.2$ of \cite{La}.

Consider the
product of Grassmannians
$B=G(k_{1},A_{1})\times\cdots \times G(k_{n},A_{n})$
and the bundle
$\mathcal{S}:= \mathcal{S}_{1}\otimes\cdots\otimes \mathcal{S}_{n}\rightarrow_{p} B$,
which is the tensor product of the tautological subspace
bundles pulled back to $B$. A point of $\mathcal{S}$ is of the form $(E_{1},\ldots,E_{n};T)$ where
$E_{j}\subset A_{j}$ is a $k_{j}$ -plane, and $T\in E_{1}\otimes\cdots\otimes E_{n}$. Consider the projection
$q: \mathcal{S}\rightarrow A_{1}\otimes\cdots\otimes A_{n}, (E_{1},\ldots,E_{n};T)\rightarrow T$. The image of $q$ is $\widehat{\mbox{\it Sub}}_{k_{1},\ldots,k_{n}}A_{1}\otimes\cdots\otimes A_{n}$. If $T$ is a smooth point in $\widehat{\mbox{\it Sub}}_{k_{1},\ldots,k_{n}}A_{1}\otimes\cdots\otimes A_{n}$, then $\dim\ T(A^{*}_{i})=k_{i}$ for all $1\leq i\leq n$ and $E_{i}=T^{t}(A^{*}_{\hat{\imath}})$ is the unique preimage of $T$ under $q$.
Thus the map $q: \mathcal{S}\rightarrow \widehat{\mbox{\it Sub}}_{k_{1},\ldots,k_{n}}A_{1}\otimes\cdots\otimes A_{n}$ is a Kempf-Weyman desingularization of $\widehat{\mbox{\it Sub}}_{k_{1},\ldots,k_{n}}A_{1}\otimes\cdots\otimes A_{n}$.

We have the following diagram:
\begin{equation} \label{weyman desing}
\begin{array}{ccccc}&&\mathcal{S}_{1}\otimes\cdots\otimes \mathcal{S}_{n}\\
& \swarrow{p} && \searrow{q} \\\ G(k_{1},A_{1})\times\cdots\times
G(k_{n},A_{n})&&&&\widehat{\mbox{\it Sub}}_{k_{1},\ldots,k_{n}}A_{1}\otimes\cdots\otimes
A_{n}
\end{array}
\end{equation}
From the Kempf-Weyman desingularization described in Chapter $7.2$ of \cite{Weyman}, and using Lemma \ref{grassman lemma} for $B$, for general $\varphi\in A'_{1}\otimes\ldots\otimes A'_{n}$, we deduce Lemma \ref{tangentspace}.
\end{proof}
\begin{prop}\label{def of many join}
Let $A_{i}$ be complex vector spaces, for all $1\leq i\leq n$; The join $\textbf{J}(\mbox{\it Sub}_{k^{1}_{1},\ldots,k^{1}_{n}}(A_{1}\otimes\cdots\otimes
A_{n}),\ldots,\mbox{\it Sub}_{k^{m}_{1},\ldots,k^{m}_{n}}(A_{1}\otimes\cdots\otimes
A_{n}))$

$(\iota)$ is non-defective if $\textrm{dim}\ A_{i}\geq \sum_{1\leq
j\leq m}k^{j}_{i},\ 1\leq j\leq m$, for each $1\leq i\leq n$;

$(\iota\iota)$ is defective if $\prod_{1\leq
i\leq n}\textrm{dim}\ A_{i}< \sum_{1\leq j\leq m}\prod_{1\leq
i\leq n}k^{j}_{i}$;

Also

$(\iota\iota\iota)$ the defect of the join is at least $\sum_{1\leq j\leq
m}\prod_{1\leq i\leq n}k^{j}_{i}-\prod_{1\leq i\leq
n}\textrm{dim}\ A_{i}$.
\end{prop}
\begin{proof}
($\iota$) If $\textrm{dim}\ A_{i}\geq \sum_{1\leq
j\leq m}k^{j}_{i},\ 1\leq j\leq m$, without loss of generality, we assume equality holds. Splitting
$A_{i}=A^{1}_{i}\oplus\cdots\oplus A^{m}_{i}$, for each $1\leq i\leq n$, such that
$\dim A^{j}_{i}=k^{j}_{i}$ and $\dim
A_{i}=\sum_{1\leq j\leq m}k^{j}_{i}$, and further
taking $\varphi_{t}\in A^{t}_{1}\otimes\cdots\otimes
A^{t}_{n}, 1\leq t\leq m$, it follows from Lemma
\ref{tangentspace} that
\begin{align*}
&\hat{T}_{\varphi_{t}}(\widehat{\mbox{\it Sub}}_{k^{t}_{1},\ldots,k^{t}_{n}}(A_{1}\otimes...\otimes
A_{n}))\\&=(A^{t}_{1}\otimes\cdots\otimes A^{t}_{n})\bigoplus_{1\leq
s\leq n} ((A^{1}_{s}\oplus\cdots\oplus
\widehat{A^{t}_{s}}\oplus\cdots\oplus A^{m}_{s})\otimes
\varphi_{t}(A^{t*}_{s})
)\\& \subset (A^{t}_{1}\otimes\cdots\otimes
A^{t}_{n})\bigoplus_{1\leq s\leq n} ((A^{1}_{s}\oplus\cdots\oplus
\widehat{A^{t}_{s}}\oplus\ldots\oplus A^{m}_{s})\otimes
A^{t}_{1}\otimes\cdots\otimes \widehat{A^{t}_{s}}\otimes\cdots\otimes
A^{t}_{n}).
\end{align*}
Since
\begin{align*}
&(A^{t}_{1}\otimes\cdots\otimes A^{t}_{n})\bigoplus_{1\leq s\leq n}
((A^{1}_{s}\oplus\cdots\oplus
\widehat{A^{t}_{s}}\oplus\cdots\oplus A^{m}_{s})\otimes
A^{t}_{1}\otimes\cdots\otimes
\widehat{A^{t}_{s}}\otimes\ldots\otimes A^{t}_{n})\\ &\bigcap
\sum_{t'\neq t}((A^{t'}_{1}\otimes\cdots\otimes
A^{t'}_{n})\bigoplus_{1\leq s\leq n} ((A^{1}_{s}\oplus\cdots\oplus
\widehat{A^{t'}_{s}}\oplus\cdots\oplus A^{m}_{s})\otimes
A^{t'}_{1}\otimes\cdots\otimes
\widehat{A^{t'}_{s}}\otimes\cdots\otimes A^{t'}_{n}))\\&=\{0\},
\end{align*}
we deduce
\begin{align*}
\hat{T}_{\varphi_{t}}(\widehat{\mbox{\it Sub}}_{k^{t}_{1},\ldots,k^{t}_{n}}(A_{1}\otimes\cdots\otimes
A_{n}))\bigcap \sum_{1\leq t'\neq t\leq m}
\hat{T}_{\varphi_{t'}}(\widehat{\mbox{\it Sub}}_{k^{t'}_{1},\ldots,k^{t'}_{n}}(A_{1}\otimes\cdots\otimes
A_{n}))=\{0\}.
\end{align*}

Using Theorem \ref{ter lem}, we have
\begin{align*}
&\hat{T}_{\sum_{1\leq t\leq m}
\varphi_{t}}(\textbf{J}(\widehat{\mbox{\it Sub}}_{k^{1}_{1},\ldots,k^{1}_{n}}(A_{1}\otimes\cdots\otimes
A_{n}),\ldots,\widehat{\mbox{\it Sub}}_{k^{m}_{1},\ldots,k^{m}_{n}}(A_{1}\otimes\cdots\otimes
A_{n})))\\=&\bigoplus_{1\leq t\leq m}\hat{T}_{\varphi_{t}}(\widehat{\mbox{\it Sub}}_{k^{t}_{1},\ldots,k^{t}_{n}}(A_{1}\otimes\cdots\otimes
A_{n})).
\end{align*}
Thus $\textbf{J}(\mbox{\it Sub}_{k^{1}_{1},\ldots,k^{1}_{n}}(A_{1}\otimes\cdots\otimes
A_{n}),\ldots,Sub_{k^{m}_{1},\ldots,k^{m}_{n}}(A_{1}\otimes\cdots\otimes
A_{n}))$ is non-defective.

(Proof of $\iota\iota-\iota\iota\iota$) When $\prod_{1\leq i\leq n}\textrm{dim}\ A_{i}< \sum_{1\leq
j\leq m}\prod_{1\leq i\leq n}k^{j}_{i}$, we have
\begin{align*}
\hat{T}_{\varphi_{t}}(\mbox{\it Sub}_{k^{t}_{1},\ldots,k^{t}_{n}}(A_{1}\otimes\cdots\otimes
A_{n}))=(A^{t}_{1}\otimes\cdots\otimes A^{t}_{n})+\sum_{1\leq s\leq
n} (A^{s}\otimes \varphi_{t}(A^{s*})),
\end{align*}
where $\varphi_{t}\in A^{t}_{1}\otimes\cdots\otimes
A^{t}_{n}, 1\leq t\leq m$.

And there exists $t_{1},t_{2}\in \{1, \ldots, m\}$, such that
\begin{align*}
(A^{t_{1}}_{1}\otimes\cdots\otimes A^{t_{1}}_{n})\cap (A^{t_{2}}_{1}\otimes\cdots\otimes A^{t_{2}}_{n})\neq\{0\},
\end{align*}
which implies
\begin{align*}
\hat{T}_{\varphi_{t_{1}}}(\mbox{\it Sub}_{k^{t_{1}}_{1},\ldots,k^{t_{1}}_{n}}(A_{1}\otimes\cdots\otimes
A_{n}))\cap \hat{T}_{\varphi_{t_{2}}}(\mbox{\it Sub}_{k^{t_{2}}_{1},\ldots,k^{t_{2}}_{n}}(A_{1}\otimes\cdots\otimes
A_{n}))\neq\{0\}.
\end{align*}
Therefore,
$\textbf{J}(\mbox{\it Sub}_{k^{1}_{1},\ldots,k^{1}_{n}}(A_{1}\otimes\cdots\otimes
A_{n}),\ldots,\mbox{\it Sub}_{k^{m}_{1},\ldots,k^{m}_{n}}(A_{1}\otimes\cdots\otimes
A_{n}))$ is defective with defect at least $\sum_{1\leq j\leq
m}\prod_{1\leq i\leq n}k^{j}_{i}-\prod_{1\leq i\leq
n}\dim A_{i}$.
\end{proof}
\begin{prop}\label{defect of 2 join}
Let $A,B$ and $C$ be complex vector spaces of dimensions $a,b,c$ respectively. If
\begin{align*}
&a'\leq (b-b')(c-c'),\ b'\leq (a-a')(c-c'),\
c'\leq (a-a')(b-b'),\\&a''\leq (b-b'')(c-c''),\ b''\leq (a-a'')(c-c''),\ c''\leq
(a-a'')(b-b''),
\end{align*}
then $\textbf{J}(Sub_{a',b',c'}(A\otimes B\otimes
C),Sub_{a'',b'',c''}(A\otimes B\otimes C))$
has defect $(a'+a''-a)^{+}(b'+b''-b)^{+}(c'+c''-c)^{+}$,
where $x^{+}=x$ if $x\geq 0$ and $0$ if $x<0$.
\end{prop}
For the proof of this Proposition, we need some preliminary considerations.

Let $A, B$ and $C$ be three complex vector spaces, of dimensions $a, b, c$ respectively, and further let $A$ be sum of two spaces $E_{A}$ and $F_{A}$, of
dimension $a'$, $a''$ respectively and $E_{A}\cap F_{A}=A_{0}$. Let $A_{1}$ and
$A_{2}$ respectively denote choices of complementary spaces in $E_{A}$
and $F_{A}$ respectively. The vector spaces $E_{B},\ F_{B},\
B_{0},\ B_{1},\ B_{2},$ and $E_{C},\ F_{C},\ C_{0},\ C_{1},\
C_{2}$ are defined in a similar manner. That is:
\begin{align}\label{3 vector space}
A&=A_{1}\oplus A_{0}\oplus A_{2},\ B=B_{1}\oplus B_{0}\oplus
B_{2},\
C=C_{1}\oplus C_{0}\oplus C_{2}.\nonumber\\
E_{A}&=A_{1}\oplus A_{0},\ F_{A}=A_{2}\oplus A_{0},\
\textrm{dim}\ E_{A}=a',\ \textrm{dim}\ F_{A}=a'',\\
E_{B}&=B_{1}\oplus B_{0},\ F_{B}=B_{2}\oplus B_{0},\
\textrm{dim}\ E_{B}=b',\ \textrm{dim}\ F_{B}=b'',\nonumber\\
E_{C}&=C_{1}\oplus C_{0},\ F_{C}=C_{2}\oplus C_{0},\ \textrm{dim}\
E_{C}=c',\ \textrm{dim}\ F_{C}=c''\nonumber.
\end{align}
Note that $a'+a''=a$, if and only if $A_{0}$ is ${0}$, and similarly for $B_{0},C_{0}$.
\begin{rem}\label{def 2 spec}
As a special case of Proposition \ref{def of many join}, if $a'+a''\leq a\ \textrm{and}\  b'+b''\leq b\ \textrm{and}\
c'+c''\leq c$, $\textbf{J}(Sub_{a',b',c'}(A\otimes B\otimes
C),Sub_{a'',b'',c''}(A\otimes B\otimes C))$ is non-defective.
\end{rem}
\begin{lem}\label{mainlemma}
There exist rational maps
\begin{align*}
f: E\otimes V\rightarrow G(e,V),
\end{align*}
where $\dim E=e$, such that for $\varphi\in E\otimes V$ and $\varphi: E^{*}\rightarrow V$ is injective, we have
\begin{align*}
f(\varphi)=\varphi(E^{*})\subset V,
\end{align*}
and the open subset
\begin{align*}
U:=\{\varphi|\varphi:E^{*}\rightarrow V\  \rm is\
\rm injective\}
\end{align*}
is the locus where $f$ is regular.
\end{lem}

\begin{proof}
Let $\dim V=v$. The image $f(u)$ is the $GL(e)$-orbit space of $
\textrm{Mat}(e,v)$, where $\textrm{Mat}(e,v)$ denotes matrices of size $e\times v$. For example, when $I= \{1,\ldots,e\}, X
\in \textrm{Mat}(e,v)$, each orbit in the affine open set is
uniquely represented by a matrix

\begin{align*}
(X_{I})^{-1}X=
\begin{bmatrix}
1 & 0 & \ldots & 0 & * &\ldots & * \\
0 & 1 & \ldots & 0 & * & \ldots & *\\
\vdots & \vdots & \ddots & \vdots &\vdots & \vdots & \vdots \\
0 & \ldots & 0 & 1 & * & \ldots & *
\end{bmatrix}
\end{align*}
in which the $e(v-e)$ entries $*$ serve as coordinates on
$\mathbb{C}^{e(v-e)}$. Note that each $*$ is a $GL(a')$-
invariant rational form on Mat$(e,v)$. This orbit coincides
with a $e$-dimensional subspaces of a fixed $v$-dimensional
vector space, which is $G(e,v)$. Therefore $f(U)\subseteq G(e,v)$ is open.
\end{proof}

\subsection{Proof of Proposition \ref{defect of 2 join}}
By \eqref{3 vector space} and using Lemma \ref{tangentspace}, for general $\varphi_{E}\in
E_{A}\otimes E_{B}\otimes E_{C},\ \varphi_{F}\in F_{A}\otimes
F_{B}\otimes F_{C}$, we have
\begin{align}\label{tan E}
&\hat{T}_{\varphi_{E}}(\widehat{Sub}_{a',b',c'}(A\otimes B\otimes
C))=E_{A}\otimes E_{B}\otimes E_{C}\\& \oplus A_{2}\otimes
\varphi_{E}(E^{*}_{A})(\subset A_{2}\otimes E_{B}\otimes E_{C})\nonumber\\&
\oplus B_{2}\otimes \varphi_{1}(E^{*}_{B})(\subset B_{2}\otimes
E_{A}\otimes E_{C})\nonumber\\&
\oplus C_{2}\otimes \varphi_{1}(E^{*}_{C})(\subset C_{2}\otimes E_{A}\otimes E_{B}),\nonumber
\end{align}
and similarly
\begin{align}\label{tan F}
&\hat{T}_{\varphi_{F}}(\widehat{Sub}_{a'',b'',c''}(A\otimes B\otimes
C))=F_{A}\otimes F_{B}\otimes F_{C}\\& \oplus A_{1}\otimes
\varphi_{F}(F^{*}_{A})(\subset A_{1}\otimes F_{B}\otimes F_{C})\nonumber\\&
\oplus B_{1}\otimes \varphi_{F}(F^{*}_{B})(\subset B_{1}\otimes
F_{A}\otimes F_{C})\nonumber\\& \oplus C_{1}\otimes
\varphi_{F}(F^{*}_{C})(\subset C_{1}\otimes F_{A}\otimes F_{B}).\nonumber
\end{align}
Therefore
\begin{align}\label{specialequation}
&A_{0}\otimes B_{0}\otimes C_{0}\subset \nonumber
\\&\hat{T}_{\varphi_{E}}(\widehat{Sub}_{a',b',c'}(A\otimes B\otimes C))
\cap \hat{T}_{\varphi_{F}}(\widehat{Sub}_{a'',b'',c''}(A\otimes B\otimes C))\\
&\subset (A_{0}\otimes B_{0}\otimes C_{0}) \oplus (A_{1}\otimes
B_{0}\otimes C_{0}) \oplus (A_{2}\otimes B_{0}\otimes C_{0}) \oplus
(A_{0}\otimes B_{2}\otimes C_{0})\nonumber\\
&\oplus (A_{0}\otimes B_{0}\otimes C_{1}) \oplus (A_{0}\otimes
B_{0}\otimes C_{2}) \oplus (A_{2}\otimes B_{0}\otimes C_{1})\oplus
(A_{1}\otimes B_{2}\otimes C_{0})\nonumber\\ &\oplus (A_{2}\otimes
B_{1}\otimes C_{0}) \oplus (A_{1}\otimes B_{0}\otimes C_{2}) \oplus
(A_{0}\otimes B_{1}\otimes C_{2})\oplus (A_{0}\otimes B_{2}\otimes
C_{2})\nonumber\\& \oplus (A_{0}\otimes B_{1}\otimes C_{0})\nonumber.
\end{align}
We first need to prove that the first inclusion in \eqref{specialequation} is actually an equality.
For this purpose, we want to choose sufficiently general $\varphi_{E},\
\varphi_{F}$ to avoid a possible larger intersection of \eqref{tan E} and \eqref{tan F}.

Let $p,\ p'$ be general elements in
$\hat{T}_{\varphi_{E}}(\widehat{Sub}_{a',b',c'}(A\otimes B\otimes C))
\cap \hat{T}_{\varphi_{F}}(\widehat{Sub}_{a'',b'',c''}(A\otimes B\otimes
C))$, and use \eqref{tan E} and \eqref{tan F} to represent $p,p'$ respectively as
\begin{align*}
&p= v_{0}+v_{1}+v_{2}+v_{3},\ \textrm{with} \\
&v_{0}\in E_{A}\otimes E_{B}\otimes E_{C},\ v_{1}\in A_{2}\otimes
\varphi_{E}(E^{*}_{A}), \\&v_{2}\in B_{2}\otimes
\varphi_{E}(E^{*}_{B}),\ v_{3}\in C_{2}\otimes
\varphi_{E}(E^{*}_{C});
\end{align*}
and
\begin{align*}
&p'= v'_{0}+v'_{1}+v'_{2}+v'_{3},\ \textrm{with} \\
&v'_{0}\in F_{A}\otimes F_{B}\otimes F_{C},\ v'_{1}\in
A_{1}\otimes \varphi_{F}(F^{*}_{A}), \\&v'_{2}\in B_{1}\otimes
\varphi_{F}(F^{*}_{B}),\ v'_{3}\in C_{1}\otimes
\varphi_{F}(F^{*}_{C}).
\end{align*}
From \eqref{specialequation}, we have
\begin{align*}
&v_{1}\in A_{2}\otimes (B_{0}\otimes C_{0}\oplus B_{1}\otimes
C_{0}\oplus B_{0}\otimes C_{1}),
\end{align*}
and hence
\begin{align*}
v_{1}\in (A_{2}\otimes \varphi_{E}(A^{*}))\bigcap (A_{2}\otimes
(B_{0}\otimes C_{0}\oplus B_{1}\otimes C_{0}\oplus B_{0}\otimes
C_{1})).
\end{align*}.

Now consider
\begin{align*}
&\varphi_{E}(A^{*})\subseteq E_{B}\otimes E_{C} \ \textrm{and} \
U=(B_{0}\otimes C_{0})\oplus (B_{1}\otimes C_{0})\oplus
(B_{0}\otimes C_{1}),
\end{align*}
and note that the codimension of $U$ in $E_{B}\otimes E_{C}$ is $(b-b')(c-c')\geq
a'$. We also consider the Schubert subvariety
\begin{align*}
\tau(U):= \{ E\in G(a',b'c')\mid E\cap U\neq \{0\}\},
\end{align*}
which has codimension $(b-b')(c-c')-a'+1$ in $G(a',b'c')$.

By virtue of Lemma \ref{mainlemma}, and note that in this case, $E=E_{A}$, $V=E_{B}\otimes E_{C}$,
letting $\varphi=\varphi_{E}^{A^{*}}$, for general $\varphi_{E}^{A^{*}}$, we have
\begin{align*}
f(\varphi_{E}^{A^{*}}) \cap U = \{0\}.
\end{align*}

But the image $f_{E}(\varphi_{E}^{A^{*}})$ is in the complement of variety $\tau(U)$. Therefore, we obtained in
this way a Zariski-open dense set of general $\varphi_{E}^{A^{*}}$.
In the same way, we can obtain a Zariski-open dense sets of
general $\varphi_{E}^{B^{*}},\ \varphi_{E}^{C^{*}}$. It follows that
\begin{align*}
&(A_{2}\otimes \varphi_{E}(A^{*}))\bigcap (A_{2}\otimes (B_{0}\otimes
C_{0}\oplus B_{1}\otimes C_{0}\oplus B_{0}\otimes C_{1}))= \{0\},
\end{align*}
and in consequence $v_{1}=\{0\}$. Similarly, we have
$v_{2}=\{0\},\ v_{3}=\{0\}$; for the same reason, $v'_{1}=\{0\}$, $v'_{2}=\{0\},\ v'_{3}=\{0\}$.
Taking the intersection of those $\varphi_{E}^{A^{*}},\
\varphi_{E}^{B^{*}},\ \varphi_{E}^{C^{*}}$, we obtain
$\varphi_{E}$. In the same way, we get $\varphi_{F}$,
that give rise to
\begin{align*}
p= v_{0}\in E_{A}\otimes E_{B}\otimes E_{C},
\end{align*}
and respectively
\begin{align*}
p'= v'_{0}\in F_{A}\otimes F_{B}\otimes F_{C}.
\end{align*}
Therefore, using \eqref{3 vector space}, we obtain
\begin{align*}
&\hat{T}_{\varphi_{E}}(\widehat{Sub}_{a',b',c'}(A\otimes B\otimes C))\cap
\hat{T}_{\varphi_{F}}(\widehat{Sub}_{a',b',c'}(A\otimes B\otimes
C))\\&=A_{0}\otimes B_{0}\otimes C_{0}.
\end{align*}

By virtue of \eqref{specialequation}, $A_{0}\otimes B_{0}\otimes
C_{0}$ is the intersection of
$\hat{T}_{\varphi_{E}}(\widehat{Sub}_{a',b',c'}(A\otimes
B\otimes C))$ and $\hat{T}_{\varphi_{F}}(\widehat{Sub}_{a',b',c'}(A\otimes
B\otimes C))$ for general $\varphi_{E}$ and $\varphi_{F}$, and this completes the proof of Theorem \ref{defect of 2 join}.

\section{Some Examples}
In this section, we exhibit several examples to clarify both the basic concepts introduced in the previous sections as well as the relations between them.

First, we exhibit an example of a defective secant variety for which the defect can be computed directly according to the formula provided by Proposition \ref{defect of 2 join}, although the conditions in that proposition are not justified.
\begin{ex}
The secant variety $\sigma_{2}(Sub_{2,3,6}(\mathbb{C}^3\otimes \mathbb{C}^5\otimes
\mathbb{C}^{11}))$ has defect $1$.
\end{ex}
\begin{proof}
In order to facilitate our exposition, let $A,B$ and $C$ be complex vector spaces of dimensions $3, 5, 11$ respectively.

First, note that for
$\varphi_{1}\in A'\otimes B'\otimes C'$ and $\varphi_{2}\in A''\otimes B''\otimes C''$,
where $A', A''$ are $2$ dimensional subspaces of $A$; $B',B''$ are $3$ dimensional subspaces of $B$
and $C',C''$ are $6$ dimensional subspaces of $C$.

Since $\textrm{dim}\ A'\cap A''\geq 1,\ \textrm{dim}\ B'\cap
B''\geq 1,\ \textrm{dim}\ C'\cap C''\geq 1$, from Lemma \ref{tangentspace}, we have
\begin{align*}
&\dim
\hat{T}_{\varphi_{1}}(\widehat{Sub}_{2,3,6}(\mathbb{C}^3\otimes
\mathbb{C}^5\otimes \mathbb{C}^{11}))\cap
\hat{T}_{\varphi_{2}}(\widehat{Sub}_{2,3,6}(\mathbb{C}^3\otimes
\mathbb{C}^5\otimes \mathbb{C}^{11}))\\ &\geq \textrm{dim}\
(A'\otimes B'\otimes C')\cap (A''\otimes B''\otimes C'')\geq 1.
\end{align*}
So if there exists a pair of points on $\sigma_{2}(Sub_{2,3,6}(\mathbb{C}^3\otimes \mathbb{C}^5\otimes
\mathbb{C}^{11}))$, whose tangent spaces has a one dimensional intersection, we can claim the defect is exactly $1$.

Choose now a pair of tensors $\{\varphi_{1},\varphi_{2}\}$, such that
\begin{align*}
\varphi_{1}=a_{1}\otimes (b_{1}\otimes c_{1}+b_{2}\otimes
c_{2}+b_{3}\otimes c_{3})+a_{2}\otimes (b_{2}\otimes
c_{2}-b_{3}\otimes c_{3})\in E_{A}\otimes E_{B}\otimes E_{C},
\end{align*}
\begin{align*}
\varphi_{2}=a_{3}\otimes (b_{2}\otimes c_{2}+b_{4}\otimes
c_{4}+b_{5}\otimes c_{5})+a_{2}\otimes (b_{2}\otimes
c_{2}-b_{5}\otimes c_{5})\in F_{A}\otimes F_{B}\otimes F_{C},
\end{align*}
where $\{a_{1},a_{2},a_{3}\}$,
$\{b_{1},\ldots,b_{5}\}$ and $\{c_{1},\ldots,c_{11}\}$ are fixed bases for
$A,B$ and $C$, and $E_{A}=\langle a_{1},a_{2}\rangle,\ F_{A}=\langle
a_{2},a_{3}\rangle,\ E_{B}=\langle b_{1},b_{2},b_{3}\rangle,\
F_{B}=\langle b_{2},b_{4},b_{5}\rangle,\ E_{C}=\langle
c_{1},\ldots,c_{6}\rangle,\ F_{C}=\langle
c_{2},c_{7},\ldots,c_{11}\rangle$.

It is clear that
\begin{align*}
&\hat{T}_{\varphi_{1}}(\widehat{Sub}_{2,3,6}(\mathbb{C}^3\otimes
\mathbb{C}^5\otimes \mathbb{C}^{11}))\\&=(E_{A}\otimes E_{B}\otimes
E_{C})\\ &+(A\otimes \langle b_{1}\otimes c_{1}+b_{2}\otimes
c_{2}+b_{3}\otimes c_{3}\rangle)\\ &+ (B\otimes \langle a_{1}\otimes
c_{1},a_{1}\otimes c_{2}+a_{2}\otimes c_{2},a_{1}\otimes
c_{3}-a_{2}\otimes c_{3}\rangle)\\&+ (C\otimes \langle a_{1}\otimes
b_{1},a_{1}\otimes b_{2}+a_{2}\otimes b_{2}, a_{1}\otimes
b_{3}-a_{2}\otimes b_{3}\rangle),
\end{align*}
and similarly
\begin{align*}
&\hat{T}_{\varphi_{2}}(\widehat{Sub}_{2,3,6}(\mathbb{C}^3\otimes
\mathbb{C}^5\otimes \mathbb{C}^{11}))\\&=(F_{A}\otimes F_{B}\otimes
F_{C})\\ &+ (A\otimes \langle b_{2}\otimes c_{2}+b_{4}\otimes
c_{4}+b_{5}\otimes c_{5}\rangle)\\ &+ (B\otimes \langle a_{3}\otimes
c_{4},a_{3}\otimes c_{2}+a_{2}\otimes c_{2},a_{3}\otimes
c_{5}-a_{2}\otimes c_{5}\rangle)\\&+ (C\otimes \langle a_{3}\otimes
b_{4},a_{3}\otimes b_{2}+a_{2}\otimes b_{2}, a_{3}\otimes
b_{5}-a_{2}\otimes b_{5}\rangle).
\end{align*}
Hence
\begin{align*}
\hat{T}_{\varphi_{1}}(\widehat{Sub}_{2,3,6}(\mathbb{C}^3\otimes
\mathbb{C}^5\otimes \mathbb{C}^{11}))\cap
\hat{T}_{\varphi_{2}}(\widehat{Sub}_{2,3,6}(\mathbb{C}^3\otimes
\mathbb{C}^5\otimes \mathbb{C}^{11}))=\langle a_{2}\otimes
b_{2}\otimes c_{2}\rangle.
\end{align*}
Therefore, $\varphi_{1}$ and $\varphi_{2}$ as chosen above are sufficiently general, and the defect is $1$.
\end{proof}
\begin{ex}\label{222 444}
The secant variety $\sigma_{2}(Sub_{2,2,2}(\mathbb{C}^{4}\otimes
\mathbb{C}^{4}\otimes \mathbb{C}^{4}))$ is tangentially weakly
defective, although it is non-defective.
\end{ex}
\begin{proof}
The fact that $\sigma_{2}(Sub_{2,2,2}(\mathbb{C}^{4}\otimes
\mathbb{C}^{4}\otimes \mathbb{C}^{4}))$ is not defective follows from Remark \ref{def 2 spec}. So now we pass to the
proof that $\sigma_{2}(Sub_{2,2,2}(\mathbb{C}^{4}\otimes
\mathbb{C}^{4}\otimes \mathbb{C}^{4}))$ is tangentially weakly
defective.

Let $A,B$ and $C$ be complex vector spaces of dimensions $4, 4, 4$ respectively. Choose the splitting
$A=A_{1}\oplus A_{2},\ B=B_{1}\oplus B_{2},\ C=C_{1}\oplus C_{2}$,
where each one of $A_{1},A_{2},B_{1},B_{2},C_{1},C_{2}$ has dimension $2$.

Since $\sigma_{2}(\mathbb{P}^{1}\times \mathbb{P}^{1}\times
\mathbb{P}^{1})=\mathbb{P}(\mathbb{C}^{2}\otimes \mathbb{C}^{2}\otimes \mathbb{C}^{2})$(see Theorem $5.5.1.1$ in \cite{La}), there exists a general pair $\{\varphi_{1},\varphi_{2}\}\in \widehat{Sub}_{2,2,2}(\mathbb{C}^{4}\otimes
\mathbb{C}^{4}\otimes \mathbb{C}^{4})$, such that
\begin{align*}
&\varphi_{1}=a_{1}\otimes b_{1}\otimes c_{1}+a_{2}\otimes
b_{2}\otimes c_{2},
\end{align*}
and
\begin{align*}
&\varphi_{2}=a_{3}\otimes b_{3}\otimes c_{3}+a_{4}\otimes
b_{4}\otimes c_{4},
\end{align*}
where
$\{a_{1}, a_{2}\}$, $\{b_{1}, b_{2}\}$, $\{c_{1}, c_{2}\}$ are bases for $A_{1}$, $B_{1}$, $C_{1}$ respectively and similarly $\{a_{3}, a_{4}\}$, $\{b_{3}, b_{4}\}$, $\{c_{3}, c_{4}\}$ are bases for $A_{2}$, $B_{2}$, $C_{2}$. And note that $\varphi_{1}+\varphi_{2}$ is a general point in $\sigma_{2}(Sub_{2,2,2}(\mathbb{C}^{4}\otimes
\mathbb{C}^{4}\otimes \mathbb{C}^{4}))$.
From Theorem \ref{ter lem}, we have
\begin{align*}
&\hat{T}_{\varphi_{1}+\varphi_{2}}(\sigma_{2}(\widehat{Sub}_{2,2,2}(\mathbb{C}^{4}\otimes
\mathbb{C}^{4}\otimes
\mathbb{C}^{4})))\\=&\bigoplus_{1\leq i,j,k\leq 2}\langle a_{i}\otimes
b_{j}\otimes c_{k}\rangle \bigoplus_{3\leq i\leq 4,1\leq j\leq 2}\langle a_{i}\otimes
b_{j}\otimes c_{j}\rangle \\&\bigoplus_{1\leq i\leq 2,3\leq j\leq 4}\langle a_{i}\otimes
b_{i}\otimes c_{j}\rangle \bigoplus_{3\leq j\leq 4,1\leq k\leq 2}\langle a_{k}\otimes
b_{j}\otimes c_{k}\rangle \\&\bigoplus_{3\leq i,j,k\leq 4}\langle a_{i}\otimes
b_{j}\otimes c_{k}\rangle \bigoplus_{1\leq i\leq 2,3\leq j\leq 4}\langle a_{i}\otimes
b_{j}\otimes c_{j}\rangle \\&\bigoplus_{3\leq i\leq 4,1\leq j\leq 2}\langle a_{i}\otimes
b_{i}\otimes c_{j}\rangle \bigoplus_{1\leq j\leq 2,3\leq k\leq 4}\langle a_{k}\otimes
b_{j}\otimes c_{k}\rangle.
\end{align*}

Define $\psi=a_{1}\otimes b_{1}\otimes c_{1}+a_{3}\otimes
b_{3}\otimes c_{3}$, which is a third general point in
$\widehat{Sub}_{2,2,2}\mathbb{C}^{4}\otimes \mathbb{C}^{4}\otimes
\mathbb{C}^{4}$, and note that
\begin{align}\label{morm 222}
&\hat{T}_{\psi}(\widehat{Sub}_{2,2,2}(\mathbb{C}^{4}\otimes
\mathbb{C}^{4}\otimes \mathbb{C}^{4}))\nonumber\\=&\bigoplus_{i,j,k\in
\{1,3\}}\langle a_{i}\otimes b_{j}\otimes
c_{k}\rangle\bigoplus_{i=2,4,j=1,3}\langle a_{i}\otimes
b_{j}\otimes c_{j}\rangle\\&\bigoplus_{j=2,4,i=1,3}\langle
b_{j}\otimes a_{i}\otimes c_{i}\rangle \bigoplus_{k=2,4,\
i=1,3}\langle c_{k}\otimes a_{i}\otimes b_{i}\rangle.\nonumber
\end{align}
It is straightforward to compute that
\begin{align*}
&\hat{T}_{\varphi_{1}}(\widehat{Sub}_{2,2,2}(\mathbb{C}^{4}\otimes
\mathbb{C}^{4}\otimes \mathbb{C}^{4}))+
\hat{T}_{\varphi_{2}}(\widehat{Sub}_{2,2,2}(\mathbb{C}^{4}\otimes
\mathbb{C}^{4}\otimes \mathbb{C}^{4}))\\& \supset
\hat{T}_{\psi}(\widehat{Sub}_{2,2,2}(\mathbb{C}^{4}\otimes
\mathbb{C}^{4}\otimes \mathbb{C}^{4})),\nonumber
\end{align*}
This implies
$\sigma_{2}(Sub_{2,2,2}(\mathbb{C}^{4}\otimes
\mathbb{C}^{4}\otimes \mathbb{C}^{4}))$ is tangentially weakly-defective.
\end{proof}
\begin{rem}
Although tangentially weakly defective does not imply
non-uniqueness, the decomposition is not unique here. The reason
is trivial:
\begin{align*}
&(a_{1}\otimes b_{1}\otimes c_{1}+a_{2}\otimes b_{2}\otimes
c_{2})+(a_{3}\otimes b_{3}\otimes c_{3}+a_{4}\otimes b_{4}\otimes
c_{4})\\=&(a_{1}\otimes b_{1}\otimes c_{1}+a_{3}\otimes
b_{3}\otimes c_{3})+(a_{2}\otimes b_{2}\otimes c_{2}+a_{4}\otimes
b_{4}\otimes c_{4}).
\end{align*}
\end{rem}

\begin{ex}\label{weak by not tan-weak}
The secant variety $\sigma_{2}(Sub_{1,2,2}(\mathbb{C}^{2}\otimes \mathbb{C}^{4}\otimes
\mathbb{C}^{4}))$ is weakly defective, but not tangentially
weakly defective.
\end{ex}
\begin{proof}
Part $1$: Let $A,\ B$ and $C$ denote complex vector spaces of dimensions $2, 4, 4$ respectively.

We need to prove that for any general hyperplane $H$ tangent to $\widehat{Sub}_{1,2,2}(\mathbb{C}^{2}\otimes \mathbb{C}^{4}\otimes
\mathbb{C}^{4})$ at a general pair of points $\{\varphi_{1},\varphi_{2}\}$
 is also tangent to
$\widehat{Sub}_{1,2,2}(\mathbb{C}^{2}\otimes \mathbb{C}^{4}\otimes
\mathbb{C}^{4})$ at some general point $\psi$, satisfying $[\psi] \neq [\varphi_{1}],[\varphi_{2}]$.

Choose now general points $\varphi_{1},\varphi_{2}\in \widehat{Sub}_{1,2,2}(\mathbb{C}^{2}\otimes
\mathbb{C}^{4}\otimes \mathbb{C}^{4})$, and  $\varphi_{1}+\varphi_{2}\in \hat{\sigma}_{2}(Sub_{1,2,2}(\mathbb{C}^{2}\otimes \mathbb{C}^{4}\otimes
\mathbb{C}^{4}))$ is also general. Without loss of generality, we assume
\begin{align*}
&\varphi_{1}=a_{1}\otimes (b_{1}\otimes c_{1}+b_{2}\otimes
c_{2}),\\
&\varphi_{2}=a_{2}\otimes (b_{3}\otimes c_{3}+b_{4}\otimes c_{4}),
\end{align*}
where $A=\langle a_{1},a_{2}\rangle$, $B=\langle b_{1},\cdots, b_{4}\rangle$ and $C=\langle c_{1},\cdots, c_{4}\rangle$.
Note that $A=A_{1}\oplus
A_{2},\ B=B_{1}\oplus B_{2},\ C=C_{1}\oplus C_{2}$, where $\{a_{1}\}$, $\{b_{1},\ b_{2}\}$, $\{c_{1}, c_{2}\}$ are bases for $A_{1}$, $B_{1}$, $C_{1}$, respectively, and similarly
$\{a_{2}\}$, $\{b_{3},\ b_{4}\}$, $\{c_{3}, c_{4}\}$ are bases for $A_{2}$, $B_{2}$, $C_{2}$.

For $\varphi_{p}\in A_{p}\otimes B_{p}\otimes C_{p},p=1,2$, we have
\begin{align}\label{tan 122}
&\hat{T}_{\varphi_{p}}(\widehat{Sub}_{1,2,2}(\mathbb{C}^{2}\otimes
\mathbb{C}^{4}\otimes \mathbb{C}^{4})\\&=(A_{p}\otimes B\otimes C_{p})+(A_{p}\otimes B_{p}\otimes C)+(A\otimes
\varphi_{p}(A^{*}_{p})),\nonumber
\end{align}
and
\begin{align*}
&\hat{T}^{\bot}_{\varphi_{p}}(\widehat{Sub}_{1,2,2}(\mathbb{C}^{2}\otimes
\mathbb{C}^{4}\otimes
\mathbb{C}^{4})=(A^{\bot}_{p}\otimes
B^{\bot}_{p}\otimes C^{*})\oplus (A^{*}\otimes B_{p}^{\bot}\otimes
C^{\bot}_{p})\\&\oplus (A_{p}^{\bot}\otimes B^{*}\otimes
C^{\bot}_{p})\oplus (A_{p}^{\bot}\otimes
(\varphi_{p}(A^{*}_{p})^{\bot}\cap (B^{*}_{p}\otimes
C^{*}_{p}))).\nonumber
\end{align*}
Then using Theorem \ref{ter lem}, we have
\begin{align*}
&\hat{T}^{\bot}_{
\varphi_{1}+\varphi_{2}}(\sigma_{2}(\widehat{Sub}_{1,2,2}(\mathbb{C}^{2}\otimes
\mathbb{C}^{4}\otimes
\mathbb{C}^{4}))\\=&\hat{T}^{\bot}_{\varphi_{1}}(\widehat{Sub}_{1,2,2}(\mathbb{C}^{2}\otimes
\mathbb{C}^{4}\otimes
\mathbb{C}^{4})\cap \hat{T}^{\bot}_{\varphi_{2}}(\widehat{Sub}_{1,2,2}(\mathbb{C}^{2}\otimes
\mathbb{C}^{4}\otimes
\mathbb{C}^{4})\\=&(A_{2}^{*}\otimes
(\varphi_{1}(A^{*}_{1})^{\bot}\cap (B^{*}_{1}\otimes
C^{*}_{1})))\oplus (A_{1}^{*}\otimes
(\varphi_{2}(A^{*}_{2})^{\bot}\cap (B^{*}_{2}\otimes
C^{*}_{2}))),
\end{align*}
which implies
\begin{align}\label{norm 122}
&\hat{T}^{\bot}_{\varphi_{1}+\varphi_{2}}(\sigma_{2}(\widehat{Sub}_{1,2,2}(\mathbb{C}^{2}\otimes
\mathbb{C}^{4}\otimes
\mathbb{C}^{4})))\nonumber\\=&\langle a^{*}_{2}\otimes
b^{*}_{1}\otimes c^{*}_{2},\ a^{*}_{2}\otimes b^{*}_{2}\otimes
c^{*}_{1},\ a^{*}_{2}\otimes b^{*}_{1}\otimes
c^{*}_{1}-a^{*}_{2}\otimes b^{*}_{2}\otimes
c^{*}_{2},\\&a^{*}_{1}\otimes b^{*}_{4}\otimes
c^{*}_{3},\ a^{*}_{1}\otimes b^{*}_{3}\otimes
c^{*}_{4},\ a^{*}_{1}\otimes b^{*}_{4}\otimes
c^{*}_{4}-a^{*}_{1}\otimes b^{*}_{3}\otimes c^{*}_{3}\rangle.\nonumber
\end{align}
Due to \eqref{norm 122}, every hyperplane tangent to $\widehat{Sub}_{1,2,2}(\mathbb{C}^{2}\otimes
\mathbb{C}^{4}\otimes \mathbb{C}^{4})$ at $\varphi_{1}$ and $\varphi_{2}$ is of the form
\begin{align*}
H&=a^{*}_{2}\otimes (\lambda_{1}b^{*}_{1}\otimes
c^{*}_{2}+\lambda_{2}
 b^{*}_{2}\otimes
c^{*}_{1}+\lambda_{3} (b^{*}_{1}\otimes c^{*}_{1}-b^{*}_{2}\otimes
c^{*}_{2}))\\&+a^{*}_{1}\otimes (\mu_{1}b^{*}_{4}\otimes
c^{*}_{3}+\mu_{2} b^{*}_{3}\otimes c^{*}_{4}+\mu_{3}
(b^{*}_{4}\otimes c^{*}_{4}-b^{*}_{3}\otimes c^{*}_{3})),
\end{align*}
where all of $\lambda_{i},\mu_{j},1\leq i,j\leq3$ are not zero.

It is straightforward to calculate that $H$ is tangent to
$\widehat{Sub}_{1,2,2}(\mathbb{C}^{2}\otimes \mathbb{C}^{4}\otimes
\mathbb{C}^{4})$ at $\psi$, where
\begin{align*}
\psi=a_{1}\otimes (-\lambda_{2}(b_{1}\otimes c_{2})+\lambda_{1}
 (b_{2}\otimes c_{1})+\lambda_{3} (b_{1}\otimes c_{1}+b_{2}\otimes
c_{2})),
\end{align*}
clearly $[\psi]\neq [\varphi_{1}],[\varphi_{2}]$.

This concludes the proof that
$\sigma_{2}(Sub_{1,2,2}(\mathbb{C}^{2}\otimes \mathbb{C}^{4}\otimes
\mathbb{C}^{4})$) is weakly defective.

We pass now to the proof that $\sigma_{2}(Sub_{1,2,2}(\mathbb{C}^{2}\otimes \mathbb{C}^{4}\otimes
\mathbb{C}^{4})$) is not tangentially weakly defective.

Part $2$:
Let $\psi=a'\otimes (b'\otimes c'+b''\otimes c'')\in A'\otimes
B'\otimes C'$ be a general point in
$Sub_{1,2,2}(\mathbb{C}^{2}\otimes \mathbb{C}^{4}\otimes
\mathbb{C}^{4})$, where $A'=\langle a' \rangle$, $B'=\langle b', b''\rangle$ and $C'=\langle c', c''\rangle$; obviously we have
\begin{align}\label{tan psi}
&\hat{T}_{\psi}(\widehat{Sub}_{1,2,2}(\mathbb{C}^{2}\otimes
\mathbb{C}^{4}\otimes \mathbb{C}^{4}))\\&=(A'\otimes B\otimes C')+
(A'\otimes B'\otimes C)+(A\otimes \psi(A'^{*})).\nonumber
\end{align}
First, let $\{\varphi_{1},\varphi_{2}\}\in \widehat{Sub}_{1,2,2}(\mathbb{C}^{2}\otimes
\mathbb{C}^{4}\otimes \mathbb{C}^{4})$. Without loss of generality, we can consider the general pair $\{\varphi_{1},\varphi_{2}\}$ as in Part $1$.

Also according to Remark \ref{nor terr},
\begin{align}\label{122 def tan}
&\hat{T}_{\varphi_{1}}(\widehat{Sub}_{1,2,2}(\mathbb{C}^{2}\otimes
\mathbb{C}^{4}\otimes \mathbb{C}^{4}))+
\hat{T}_{\varphi_{2}}(\widehat{Sub}_{1,2,2}(\mathbb{C}^{2}\otimes
\mathbb{C}^{4}\otimes \mathbb{C}^{4}))\\& \supset
\hat{T}_{\psi}(\widehat{Sub}_{1,2,2}(\mathbb{C}^{2}\otimes
\mathbb{C}^{4}\otimes \mathbb{C}^{4})),\nonumber
\end{align}
is equivalent to that
\begin{align}\label{122 def tan rever}
&\hat{T}^{\bot}_{\varphi_{1}}(\widehat{Sub}_{1,2,2}(\mathbb{C}^{2}\otimes
\mathbb{C}^{4}\otimes \mathbb{C}^{4}))\cap
\hat{T}^{\bot}_{\varphi_{2}}(\widehat{Sub}_{1,2,2}(\mathbb{C}^{2}\otimes
\mathbb{C}^{4}\otimes \mathbb{C}^{4}))\\& \subset
\hat{T}^{\bot}_{\psi}(\widehat{Sub}_{1,2,2}(\mathbb{C}^{2}\otimes
\mathbb{C}^{4}\otimes \mathbb{C}^{4})).\nonumber
\end{align}
Hence we need to prove that these inclusions imply that $[\psi]$ is either $[\varphi_{1}]$ or $[\varphi_{2}]$.

Express $a',c'$ as
\begin{align*}
a'=x_{1}a_{1}+x_{2}a_{2},\ c'=z_{1}c_{1}+\ldots+z_{4}c_{4}.
\end{align*}
and we first treat the case when $x_{1},\ x_{2}$ are both nonzero.
A hyperplane
\begin{align*}
H_{1}=a^{*}_{2}\otimes b^{*}_{1}\otimes c^{*}_{2}
\end{align*}
and respectively
\begin{align*}
H_{2}=a^{*}_{2}\otimes b^{*}_{2}\otimes c^{*}_{1}
\end{align*}
is tangent to $\widehat{Sub}_{1,2,2}(\mathbb{C}^{2}\otimes
\mathbb{C}^{4}\otimes \mathbb{C}^{4})$ at $\psi$, only if, using the symbol $\vdash$ introduced in \eqref{rank 1},
we have
\begin{align*}
a^{*}_{2}\otimes b^{*}_{1}\otimes c^{*}_{2}\vdash
(x_{1}a_{1}+x_{2}a_{2})\otimes b_{1}\otimes
(z_{1}c_{1}+\ldots+z_{4}c_{4})=x_{2}z_{2}=0,
\end{align*}
and respectively,
\begin{align*}
a^{*}_{2}\otimes b^{*}_{2}\otimes c^{*}_{1}\vdash
(x_{1}a_{1}+x_{2}a_{2})\otimes b_{2}\otimes
(z_{1}c_{1}+\ldots+z_{4}c_{4})=x_{2}z_{1}=0,
\end{align*}
where
\begin{align*}
(x_{1}a_{1}+x_{2}a_{2})\otimes b_{j}\otimes
(z_{1}c_{1}+\ldots+z_{4}c_{4})\in A'\otimes B\otimes C',\ j=1,2.
\end{align*}
This implies $z_{2},\ z_{1}=0$ and by symmetry, $z_{k}=0$, for $1\leq
k\leq 4$. Then $c'=0$, so $b'=0$ by symmetry; for the same reason $c''=b''=0$, hence $\psi=0$.

Next without loss of generality, we treat the case when $a'=a_{1}$; taking $H=a^{*}_{1}\otimes b^{*}_{4}\otimes
c^{*}_{3},\ a^{*}_{1}\otimes b^{*}_{3}\otimes c^{*}_{4}$, we
obtain $z_{3},\ z_{4}=0$, so $c'\in C_{1}$. By symmetry, $c''\in
C_{1}$; and for the same reason, $b',b''\in B_{1}$. Thus we have $\psi\in A_{1}\otimes
B_{1}\otimes C_{1}$. Similarly when $a'=a_{2}$, we have $\psi\in
A_{2}\otimes B_{2}\otimes C_{2}$.

From the above analysis, $H$ is tangent to
$\widehat{Sub}_{1,2,2}(\mathbb{C}^{2}\otimes \mathbb{C}^{4}\otimes
\mathbb{C}^{4})$ at $\psi$ only if $\psi\in \bigcup_{p=1,2} A_{p}\otimes
B_{p}\otimes C_{p}$. But without loss of generality, if $\psi\in A_{1}\otimes
B_{1}\otimes C_{1}$, in the previous case, \eqref{122 def tan} becomes
\begin{align*}
&\bigoplus_{p=1,2}\{(A_{p}\otimes B\otimes C_{p})+(A_{p}\otimes B_{p}\otimes C)+(A\otimes
\varphi_{p}(A^{*}_{p}))\}\\&\supset (A_{1}\otimes B\otimes C_{1})+(A_{1}\otimes B_{1}\otimes C)+(A\otimes
\psi(A^{*}_{1})),
\end{align*}
and hence
$[\psi(A^{*}_{1})]=[\varphi_{1}(A_{1}^{*})]$,
which implies $[\psi]=[\varphi_{1}]$.
By semicontinuity, this concludes the proof that $\sigma_{2}(Sub_{1,2,2}(\mathbb{C}^{2}\otimes
\mathbb{C}^{4}\otimes \mathbb{C}^{4}))$ is not tangentially
weakly-defective.
\end{proof}
\section{Proof of main result}
This section contains a proof of Theorem \ref{main results}. However, before passing to the proof, it worth noting that the geometric meaning of the
basic hypothesis of Theorem \ref{main results}, namely
\begin{align*}
\sum_{1\leq r\leq R}(J\cdot L_{r}+L_{r}\cdot (K-L_{r})+(I-1))<IJK,
\end{align*}
is equivalent to the join $\textbf{J}(\mbox{\it Sub}_{1,L_{1},L_{1}}(\mathbb{C}^{I}\otimes
\mathbb{C}^{J}\otimes
\mathbb{C}^{K}),\ldots,\mbox{\it Sub}_{1,L_{R},L_{R}}(\mathbb{C}^{I}\otimes
\mathbb{C}^{J}\otimes \mathbb{C}^{K}))$ not filling its ambient space.

It is relatively easy to
determine cases where tensors have at most a finite number of decompositions.
\subsection{Proof that Condition A implies partial uniqueness}
Let $A,\ B$ and $C$ be complex vector spaces of dimensions $I, J, K$ respectively.
Choose general $\varphi_{p}\in \widehat{\mbox{\it Sub}}_{1,L_{p},L_{p}}(\mathbb{C}^I\otimes
\mathbb{C}^J\otimes \mathbb{C}^{K}),1\leq p\leq R$. Without loss of generality, we assume
\begin{align*}
\varphi_{p}=(a_{1}+\lambda^{p}a_{2})\otimes (b_{p,1}\otimes
c_{p,1}+b_{p,2}\otimes c_{p,2}+\cdots+b_{p,L_{p}}\otimes
c_{p,L_{p}})\in A_{p}\otimes B_{p}\otimes C_{p},
\end{align*}
where $\{a_{1}+\lambda^{p}a_{2}\}$ are bases for $A_{p}$, $\{b_{p,1},\ldots,b_{p,L_{p}}\}\subset \{b_{1},\ldots, b_{J}\}$,
$\{c_{p,1},\ldots,c_{p,L_{p}}\}\subset \{c_{1},\ldots, c_{K}\}$ are bases for $B_{p}$, $C_{p}$, where $1 \leq p\leq R$.

Using Lemma \ref{tangentspace}, it is straightforward to compute
\begin{align*}
&\hat{T}_{\varphi_{p}}(\widehat{\mbox{\it Sub}}_{1,L_{p},L_{p}}(\mathbb{C}^I\otimes
\mathbb{C}^J\otimes \mathbb{C}^{K}))\\=&(A\otimes
\langle\sum_{1\leq j\leq L_{p}}b_{p,j}\otimes
c_{p,j}\rangle)+ (\langle a_{1}+\lambda^{p}a_{2}\rangle\otimes
B\otimes C_{p})+ (\langle a_{1}+\lambda^{p}a_{2}\rangle\otimes
B_{p}\otimes C)
\end{align*}
and by Theorem \ref{ter lem}, we have
\begin{align}\label{def 1ll}
&\hat{T}_{\sum^{R}_{p=1}\varphi_{p}}(\hat{\textbf{J}}(\mbox{\it Sub}_{1,L_{1},L_{1}}(\mathbb{C}^{I}\otimes
\mathbb{C}^{J}\otimes
\mathbb{C}^{K}),\ldots,\mbox{\it Sub}_{1,L_{R},L_{R}}(\mathbb{C}^{I}\otimes
\mathbb{C}^{J}\otimes \mathbb{C}^{K})))\nonumber\\=&\sum^{R}_{p=1}\{(A\otimes
\langle\sum_{1\leq j\leq L_{p}}b_{p,j}\otimes
c_{p,j}\rangle)\\&+ (\langle a_{1}+\lambda^{p}a_{2}\rangle\otimes
B\otimes C_{p})+ (\langle a_{1}+\lambda^{p}a_{2}\rangle\otimes
B_{p}\otimes C)\}.\nonumber
\end{align}
We need to prove that \eqref{def 1ll} is a direct sum at least for this set of points $\{\varphi_{1},\ldots,\varphi_{R}\}$.

Since ${{J}\choose{L_{R}}}, {{K}\choose{L_{R}}}\geq R$,
we can choose bases such that for any pair $p\neq q$, there exists $b_{s}\in \{b_{p,1},\ldots,b_{p,L_{p}}\}$ not belong to $\{b_{q,1},\ldots,b_{q,L_{q}}\}$; and similarly there exists $c_{s}\in\{c_{p,1},\ldots,c_{p,L_{p}}\}$ not belong to $\{c_{q,1},\ldots,c_{q,L_{q}}\}$ for any pair $p\neq q$. Then
we have
\begin{align*}
(\sum_{1\leq j\leq L_{p}}b_{p,j}\otimes c_{p,j}) \cap (B\otimes C_{q}+B_{q}\otimes C)=\{0\}.
\end{align*}
Therefore,
\begin{align*}
&(A\otimes
\langle\sum_{1\leq j\leq L_{p}}b_{p,j}\otimes
c_{p,j}\rangle) \bigcap \\(\sum_{\forall q\neq p} \{&(A\otimes
\langle\sum_{1\leq j\leq L_{p}}b_{q,j}\otimes
c_{q,j}\rangle)+ (\langle a_{1}+\lambda^{q}a_{2}\rangle\otimes
B\otimes C_{q})\\&+ (\langle a_{1}+\lambda^{q}a_{2}\rangle\otimes
B_{q}\otimes C)\})\\&=\sum_{\forall q\neq p} \langle a_{1}+\lambda^{q}a_{2}\rangle\otimes
(\langle\sum_{1\leq j\leq L_{p}}b_{p,j}\otimes
c_{p,j}\rangle \cap (B\otimes C_{q}+
B_{q}\otimes C))\\&=\{0\},
\end{align*}
as well as
\begin{align*}
\{&(\langle a_{1}+\lambda^{p}a_{2}\rangle\otimes B\otimes
C_{p})+ (\langle a_{1}+\lambda^{p}a_{2}\rangle\otimes
B_{p}\otimes C)\}\\&\bigcap \sum_{p\neq q} \{(\langle
a_{1}+\lambda^{q}a_{2}\rangle\otimes B\otimes C_{q})+ (\langle
a_{1}+\lambda^{q}a_{2}\rangle\otimes B_{q}\otimes C)\}=\{0\}.
\end{align*}
Thus \eqref{def 1ll} is a direct sum, and by semicontinuity, this concludes the proof of non-defectivity.
Now Condition $A$ follows from Case \ref{partial unique}.
\subsection{Proof that Condition B implies non-uniqueness}
Using Proposition \ref{def of many join}, and from Corollary \ref{partial unique}, we deduce Case $B$.

\subsection{Proof that Condition C implies generic uniqueness}
It is sufficient to prove the case $I=2, J=K=\sum^{R}_{r=1}L_{r}$. Otherwise we replace the equality in \eqref{norm1LL R} with the inclusion $\supset$.

Let $A,\ B$ and $C$ be complex vector spaces of dimensions $I, J, K$ respectively. Split $B=\bigoplus_{1\leq q\leq R}B_{q}$ and $C=\bigoplus_{1\leq r\leq R}C_{r}$, where for $1\leq q,r\leq R$,  $B_{q}$ and $C_{r}$ are of dimensions $L_{q},L_{r}$, respectively.

Choose a general set $\{\varphi_{p}\in \widehat{\mbox{\it Sub}}_{1,L_{p},L_{p}}(\mathbb{C}^I\otimes
\mathbb{C}^J\otimes \mathbb{C}^{K}):1\leq p\leq R\}$. Without loss of generality, we can assume
\begin{align*}
\varphi_{p}=(a_{1}+\lambda^{p}a_{2})\otimes (b_{p,1}\otimes
c_{p,1}+b_{p,2}\otimes c_{p,2}+\cdots+b_{p,L_{p}}\otimes
c_{p,L_{p}})\in A_{p}\otimes B_{p}\otimes C_{p},
\end{align*}
for any $1\leq p\leq R$, where $\{a_{1}+\lambda^{p}a_{2}\}$, $\{b_{p,1},\ldots,b_{p,L_{p}}\}$ and
$\{c_{p,1},\ldots,c_{p,L_{p}}\}$ are bases for $A_{p}$, $B_{p}$, $C_{p}$.
Note that for a general set $\{\varphi_{p}\in A_{p}\otimes B_{p}\otimes C_{p}:1\leq p\leq R\}$, we have
\begin{align}
&\hat{T}_{\varphi_{p}}(\widehat{\mbox{\it Sub}}_{1,L_{p},L_{p}}(\mathbb{C}^I\otimes
\mathbb{C}^J\otimes \mathbb{C}^{K}))\\=&(A\otimes
\varphi_{p}(A^{*}_{p}))+ (A_{p}\otimes
B\otimes C_{p})+ (A_{p}\otimes B_{p}\otimes C),\nonumber
\end{align}
and
\begin{align*}
&\hat{T}^{\bot}_{\varphi_{p}}(\widehat{\mbox{\it Sub}}_{1,L_{p},L_{p}}(\mathbb{C}^{I}\otimes
\mathbb{C}^{J}\otimes
\mathbb{C}^{K}))=(A^{\bot}_{p}\otimes
B^{\bot}_{p}\otimes C^{*})\oplus (A_{p}^{*}\otimes
B_{p}^{\bot}\otimes C^{\bot}_{p})\\&\oplus (A_{p}^{\bot}\otimes
B_{p}^{*}\otimes C^{\bot}_{p})\oplus (A_{p}^{\bot}\otimes
(\varphi_{p}(A^{*}_{p})^{\bot}\cap (B^{*}_{p}\otimes
C^{*}_{p}))).
\end{align*}

Then due to Theorem \ref{ter lem}, we deduce
\begin{align}\label{norm1LL R}
&\hat{T}^{\bot}_{\sum
\varphi_{p}}(\textbf{J}(\widehat{\mbox{\it Sub}}_{1,L_{1},L_{1}}(\mathbb{C}^{I}\otimes
\mathbb{C}^{J}\otimes
\mathbb{C}^{K}),\ldots,\widehat{\mbox{\it Sub}}_{1,L_{R},L_{R}}(\mathbb{C}^{I}\otimes
\mathbb{C}^{J}\otimes
\mathbb{C}^{K})))\nonumber\\=&\bigcap_{1\leq p\leq R}\widehat{T}^{\bot}_{\varphi_{p}}(\widehat{\mbox{\it Sub}}_{1,L_{p},L_{p}}(\mathbb{C}^{I}\otimes
\mathbb{C}^{J}\otimes
\mathbb{C}^{K}))\nonumber\\=&\bigoplus_{1\leq p\leq R}A_{p}^{\bot}\otimes
(\varphi_{p}(A^{*}_{p})^{\bot}\cap (B^{*}_{p}\otimes
C^{*}_{p}))\\=&\bigoplus_{1\leq p\leq R,\ j\neq
k}\langle \lambda^{p}a^{*}_{1}-a^{*}_{2}\rangle\otimes \langle
b^{*}_{p,j}\otimes c^{*}_{p,k}, b^{*}_{p,j}\otimes
c^{*}_{p,j}-b^{*}_{p,k}\otimes c^{*}_{p,k}\rangle.\nonumber
\end{align}
For any $1\leq s\leq R$, let
\begin{align}\label{psi s}
\psi_{s}=a'\otimes (b_{1}'\otimes c_{1}'+\cdots+b_{L_{s}}'\otimes
c_{L_{s}}')\in A'\otimes B'\otimes C',
\end{align}
be a general point of $\widehat{\mbox{\it Sub}}_{1,L_{s},L_{s}}\mathbb{C}^{I}\otimes \mathbb{C}^{J}\otimes
\mathbb{C}^{K}$, where $A'=\langle a' \rangle$, $B'=\langle b_{1}',\ldots,b_{L_{s}}'\rangle$ and $C'=\langle c_{1}',\ldots,c_{L_{s}}'\rangle$; note that
\begin{align}\label{tan psi D}
&\hat{T}_{\psi_{s}}(\widehat{\mbox{\it Sub}}_{1,L_{s},L_{s}}\mathbb{C}^{I}\otimes \mathbb{C}^{J}\otimes
\mathbb{C}^{K})\\&=(A'\otimes B\otimes C')+
(A'\otimes B'\otimes C)+(A\otimes \psi(A'^{*})).\nonumber
\end{align}
Also according to Remark \ref{nor terr}, the relation
\begin{align}\label{tan 1ll D}
&\hat{T}_{\varphi_{1}}(\widehat{\mbox{\it Sub}}_{1,L_{1},L_{1}}(\mathbb{C}^{I}\otimes
\mathbb{C}^{J}\otimes \mathbb{C}^{K}))+\cdots +
\hat{T}_{\varphi_{R}}(\widehat{\mbox{\it Sub}}_{1,L_{R},L_{R}}(\mathbb{C}^{I}\otimes
\mathbb{C}^{J}\otimes \mathbb{C}^{K}))\\&\supset \hat{T}_{\psi_{s}}(\widehat{\mbox{\it Sub}}_{1,L_{s},L_{s}}(\mathbb{C}^{I}\otimes
\mathbb{C}^{J}\otimes \mathbb{C}^{K})),\nonumber
\end{align}
is equivalent to
\begin{align*}
\bigcap_{1\leq p\leq R}\widehat{T}^{\bot}_{\varphi_{p}}(\widehat{\mbox{\it Sub}}_{1,L_{p},L_{p}}(\mathbb{C}^{I}\otimes
\mathbb{C}^{J}\otimes
\mathbb{C}^{K}))\subset \hat{T}^{\bot}_{\psi_{s}}(\widehat{\mbox{\it Sub}}_{1,L_{s},L_{s}}(\mathbb{C}^{I}\otimes
\mathbb{C}^{J}\otimes \mathbb{C}^{K})),
\end{align*}
Hence we need to prove that these inclusions imply that $\psi_{s}\in \{\varphi_{1},\ldots,\varphi_{R}\}$.

Denote $c'$ any one of $c_{1}',\ldots,c_{L_{s}}'$ in \eqref{psi s} and write $a',c'$ as
\begin{align*}
&a'=x_{1}a_{1}+x_{2}a_{2},\\ &c'=\sum_{1\leq h\leq L_{1}}z_{1,h}c_{1,h}+\cdots+\sum_{1\leq h\leq L_{R}}z_{R,h}c_{R,h}.
\end{align*}
We treat first the case when $(x_{1},\ x_{2})\neq
(1,\ \lambda^{p})$ for any $1\leq p\leq R$.

A general hyperplane in \eqref{norm1LL R} is a linear combination of $(\lambda^{p}a^{*}_{1}-a^{*}_{2})\otimes
b^{*}_{p,j}\otimes c^{*}_{p,k}$ and $(\lambda^{p}a^{*}_{1}-a^{*}_{2})\otimes (b^{*}_{p,j}\otimes
c^{*}_{p,j}-b^{*}_{p,k}\otimes c^{*}_{p,k})$. In particular,
\begin{align*}
(\lambda^{p}a^{*}_{1}-a^{*}_{2})\otimes b^{*}_{p,j}\otimes
c^{*}_{p,k}\in \hat{T}^{\bot}_{\psi_{s}}(\widehat{\mbox{\it Sub}}_{1,L_{s},L_{s}}(\mathbb{C}^{I}\otimes
\mathbb{C}^{J}\otimes \mathbb{C}^{K}))
\end{align*}
(i.e. is tangent to $\widehat{\mbox{\it Sub}}_{1,L_{s},L_{s}}(\mathbb{C}^{I}\otimes
\mathbb{C}^{J}\otimes \mathbb{C}^{K})$ at $\psi_{s}$) only if
\begin{align*}
&(\lambda^{p}a^{*}_{1}-a^{*}_{2})\otimes b^{*}_{p,j}\otimes
c^{*}_{p,k}\vdash \\&(x_{1}a_{1}+x_{2}a_{2})\otimes
b_{p,j}\otimes (\sum_{1\leq h\leq L_{1}}z_{1,h}c_{1,h}+\cdots+\sum_{1\leq h\leq L_{R}}z_{R,h}c_{R,h})\\&=(\lambda^{p}x_{1}-x_{2})z_{t,p_{k}}=0,
\end{align*}
where
\begin{align*}
(x_{1}a_{1}+x_{2}a_{2})\otimes b_{p,j}\otimes (\sum_{1\leq
k\leq L} z_{t,r_{k}}c_{t,r_{k}})\in A'\otimes B\otimes C'.
\end{align*}
Therefore $z_{t,h}=0,1\leq t\leq R$ for any $h$, so $c'=0$, that is $c'_{1}=\cdots=c'_{L_{s}}=0$;
for the same reason $b'_{1}=\cdots=b'_{L_{s}}=0$,
hence $\psi_{s}=0$ in this case.

Next, we consider the case when $a'=a_{1}+\lambda a_{2}$; taking a hyperplane in \eqref{norm1LL R}
\begin{align*}
H_{p}=(\lambda^{p}a^{*}_{1}-a^{*}_{2})\otimes b^{*}_{p,j}\otimes
c^{*}_{p,k},\ p\neq 1,
\end{align*}
we obtain $z_{t,h}=0,2\leq t\leq R$ for any $h$.
So we deduce
\begin{align*}
c'=\sum_{1\leq h\leq L_{1}}z_{1,h}c_{1,h}\in C_{1},
\end{align*}
for every $c'\in \{c'_{1},\ldots,c'_{L_{s}}\}$. By symmetry, we
have $b'_{1},\ldots, b'_{L_{s}}\in B_{1}$. Thus $\psi_{s}\in
A_{1}\otimes B_{1}\otimes C_{1}$.
In this case, \eqref{tan 1ll D} becomes
\begin{align*}
&\bigoplus_{1\leq p\leq R}\{(A_{p}\otimes B\otimes C_{p})+
(A_{p}\otimes B_{p}\otimes C)+(A\otimes \varphi_{p}(A_{p}^{*}))\}\\&\supset(A_{1}\otimes B\otimes C_{1})+
(A_{1}\otimes B_{1}\otimes C)+(A\otimes \psi_{s}(A_{1}^{*})),\nonumber
\end{align*}
this is valid if and only if
\begin{align*}
[\psi_{s}(A^{*}_{1})]=[\varphi_{1}(A_{1}^{*})]
\end{align*}
and then $[\psi_{s}]=[\varphi_{1}]$.

Finally, it remains to consider the case when
$a'=a_{1}+\lambda^{q} a_{2},\ 1<q\leq R$. Then clearly, we have $\psi_{s}\in A_{q}\otimes
B_{q}\otimes C_{q}$, so $[\psi_{s}]=[\varphi_{q}]$. Now by semicontinuity, we can conclude the proof that
$\textbf{J}(\mbox{\it Sub}_{1,L_{1},L_{1}}(\mathbb{C}^{I}\otimes
\mathbb{C}^{J}\otimes
\mathbb{C}^{K}),\ldots,\mbox{\it Sub}_{1,L_{R},L_{R}}(\mathbb{C}^{I}\otimes
\mathbb{C}^{J}\otimes \mathbb{C}^{K}))$ is not tangentially
weakly-defective for $I=2,\ J=K=\sum^{R}_{s=1}L_{s}$. From Corollary \ref{twd unique}, we obtain Case C.

\begin{rem}
An alternative least square algorithm for the computation of
decompositions under the above Case $C$ is given in \cite{1LL}.
\end{rem}

\subsection{Proof that Condition D implies generic uniqueness}
The proof will be nearly the same as that for condition C. (The difference is that we cannot apply Corollary \ref{twd unique}  directly in the proof of
Condition D.) As there, it is sufficient to prove the case
 $I=2$. Otherwise we change the equality in \eqref{norm 2} into the inclusion $\supset$.

Let $A,\ B$ and $C$ denote vector spaces of dimensions $I, J, K$ respectively. Split $A=A_{1}\oplus A_{2},\
B=B_{1}\oplus B_{0}\oplus B_{2}$ and $C_{1}\oplus C_{0}\oplus C_{2}$, where $A_{1},A_{2}$ are of dimension one, $B_{1}$, $B_{0}$, $B_{2}$, and $C_{1}$, $C_{0}$, $C_{2}$ are of dimension $L_{1}-l_{b}$, $l_{b}$, $L_{2}-l_{b}$, $L_{1}-l_{c}$, $l_{c}$, $L_{2}-l_{c}$, respectively, for some $0\leq l_{b},l_{c}< \min \{L_{1},L_{2}\}$.

For general $\varphi_{p}\in A_{p}\otimes (B_{p}\oplus B_{0})\otimes (C_{p}\oplus C_{0}),p=1,2$, we have
\begin{align}
&\hat{T}_{\varphi_{p}}(\widehat{\mbox{\it Sub}}_{1,L_{p},L_{p}}(\mathbb{C}^I\otimes
\mathbb{C}^J\otimes \mathbb{C}^{K}))\\=&(A\otimes
\varphi_{p}(A^{*}_{p}))+ (A_{p}\otimes
B\otimes (C_{p}\oplus C_{0}))+ (A_{p}\otimes (B_{p}\oplus B_{0})\otimes C);\nonumber
\end{align}
and also
\begin{align*}
&\hat{T}^{\bot}_{\varphi_{1}}(\widehat{\mbox{\it Sub}}_{1,L_{1},L_{1}}(\mathbb{C}^{I}\otimes
\mathbb{C}^{J}\otimes
\mathbb{C}^{K}))\nonumber\\=&(A^{*}_{2}\otimes B^{*}_{2}\otimes
C^{*})\oplus (A_{1}^{*}\otimes B_{2}^{*}\otimes C^{*}_{2})\oplus
(A_{2}^{*}\otimes (B_{1}\oplus B_{0})^{*}\otimes
C^{*}_{2})\\ &\oplus (A_{2}^{*}\otimes (\varphi_{1}(A^{*}_{1})^{\bot}\cap ((B^{*}_{1}\oplus B^{*}_{0})\otimes
(C^{*}_{1}\oplus C^{*}_{0})))),
\end{align*}
respectively
\begin{align*}
&\hat{T}^{\bot}_{\varphi_{2}}(\widehat{\mbox{\it Sub}}_{1,L_{2},L_{2}}(\mathbb{C}^{I}\otimes
\mathbb{C}^{J}\otimes
\mathbb{C}^{K}))\nonumber\\=&(A^{*}_{1}\otimes B^{*}_{1}\otimes
C^{*})\oplus (A_{2}^{*}\otimes B_{1}^{*}\otimes C^{*}_{1})\oplus
(A_{1}^{*}\otimes (B_{2}\oplus B_{0})^{*}\otimes C^{*}_{1})
\\ &\oplus (A_{1}^{*}\otimes (\varphi_{2}(A^{*}_{2})^{\bot}\cap ((B^{*}_{2}\oplus B^{*}_{0})\otimes
(C^{*}_{2}\oplus C^{*}_{0})))).
\end{align*}
Choose general points $\varphi_{p}\in \widehat{\mbox{\it Sub}}_{1,L_{p},L_{p}}(\mathbb{C}^I\otimes
\mathbb{C}^J\otimes \mathbb{C}^{K})$, for $1\leq p\leq 2$. Without loss of generality, we assume $l_{b}\geq l_{c}$, and then
\begin{align*}
\varphi_{1}&=a_{1}\otimes (b_{1,1}\otimes c_{1,1}+\cdots+ b_{1,L_{1}-l_{b}}\otimes c_{1,L_{1}-l_{b}}+b_{0,1}\otimes c_{1,L_{1}-l_{b}+1}+\cdots+b_{0,l_{b}}\otimes c_{0,l_{c}})\\&\in A_{1}\otimes (B_{1}\oplus B_{0})\otimes (C_{1}\oplus C_{0})\cong \mathbb{C}\otimes
\mathbb{C}^{L_{1}}\otimes \mathbb{C}^{L_{1}},
\end{align*}
and respectively
\begin{align*}
\varphi_{2}&=a_{2}\otimes (b_{2,1}\otimes c_{2,1}+\cdots+ b_{2,L_{1}-l_{b}}\otimes c_{2,L_{1}-l_{b}}+b_{0,1}\otimes c_{2,L_{1}-l_{b}+1}+\cdots+b_{0,l_{b}}\otimes c_{0,l_{c}})\\&\in A_{2}\otimes (B_{2}\oplus B_{0})\otimes (C_{2}\oplus C_{0})\cong \mathbb{C}\otimes
\mathbb{C}^{L_{2}}\otimes \mathbb{C}^{L_{2}},
\end{align*}
where $A_{i}=\langle a_{i}\rangle (i=1,2)$; $\{b_{0,1},\ldots,b_{0,l_{b}}\}$, $\{b_{1,1},\ldots,b_{1,L_{1}-l_{b}}\}$, $\{b_{2,1},\ldots,b_{2,L_{2}-l_{b}}\}$, $\{c_{0,1},\ldots,c_{0,l_{c}}\}$, $\{c_{1,1},\ldots,c_{1,L_{1}-l_{c}}\}$ and $\{c_{2,1},\ldots,c_{2,L_{2}-l_{c}}\}$
are bases for $B_{0}$, $B_{1}$, $B_{2}$, $C_{0}$, $C_{1}$ and $C_{2}$ respectively, where $J+l_{b}=L_{1}+L_{2}$ and $K+l_{c}=L_{1}+L_{2}$.

Using these bases, we obtain
\begin{align}\label{norm 2}
&\hat{T}^{\bot}_{\varphi_{1}}(\widehat{\mbox{\it Sub}}_{1,L_{1},L_{1}}(\mathbb{C}^{I}\otimes
\mathbb{C}^{J}\otimes \mathbb{C}^{K}))\cap
\hat{T}^{\bot}_{\varphi_{2}}(\widehat{\mbox{\it Sub}}_{1,L_{2},L_{2}}(\mathbb{C}^{I}\otimes
\mathbb{C}^{J}\otimes \mathbb{C}^{K}))\nonumber\\=&(A_{2}^{*}\otimes
(\varphi_{1}(A^{*}_{1})^{\bot}\cap (B_{1}^{*}\otimes
C^{*}_{1})))\oplus (A_{1}^{*}\otimes
(\varphi_{2}(A^{*}_{2})^{\bot}\cap (B_{2}^{*}\otimes C^{*}_{2})))
\nonumber\\=&\bigoplus_{j\neq k}\langle a^{*}_{1}\otimes
b^{*}_{2,j}\otimes c^{*}_{2,k},a^{*}_{2}\otimes
b^{*}_{1,j}\otimes c^{*}_{1,k},\\& a^{*}_{1}\otimes
(b^{*}_{2,j}\otimes c^{*}_{2,j}-
b^{*}_{2,k}\otimes c^{*}_{2,k}),a^{*}_{2}\otimes
(b^{*}_{1,j}\otimes c^{*}_{1,j}-
b^{*}_{1,k}\otimes c^{*}_{1,k})\rangle.\nonumber
\end{align}
Let
\begin{align}\label{psi D}
\psi_{s}=a'\otimes (b_{1}'\otimes c_{1}'+\cdots+b_{L_{s}}'\otimes
c_{L_{s}}')\in A'\otimes B'\otimes C',
\end{align}
be a general point in one of $\widehat{\mbox{\it Sub}}_{1,L_{s},L_{s}}(\mathbb{C}^{I}\otimes \mathbb{C}^{J}\otimes
\mathbb{C}^{K}),\ s=1,2$, where $A'=\langle a' \rangle$, $B'=\langle b_{1}',\ldots,b_{L_{s}}'\rangle$ and $C'=\langle c_{1}',\ldots,c_{L_{s}}'\rangle$.
Using Lemma \ref{tangentspace}, we have
\begin{align}\label{tan psi C}
&\hat{T}_{\psi_{s}}(\widehat{\mbox{\it Sub}}_{1,L_{s},L_{s}}(\mathbb{C}^{I}\otimes \mathbb{C}^{J}\otimes
\mathbb{C}^{K}))\\&=(A'\otimes B\otimes C')+
(A'\otimes B'\otimes C)+(A\otimes \psi(A'^{*})).\nonumber
\end{align}
If
\begin{align}\label{phi1+phi2}
[\varphi_{1}+\varphi_{2}]=[\psi_{1}+\psi_{2}],
\end{align}
we have
\begin{align}\label{tan 1ll C}
&\hat{T}_{\varphi_{1}}(\widehat{\mbox{\it Sub}}_{1,L_{1},L_{1}}(\mathbb{C}^{I}\otimes
\mathbb{C}^{J}\otimes \mathbb{C}^{K}))+
\hat{T}_{\varphi_{2}}(\widehat{\mbox{\it Sub}}_{1,L_{2},L_{2}}(\mathbb{C}^{I}\otimes
\mathbb{C}^{J}\otimes \mathbb{C}^{K}))\\&\supset \hat{T}_{\psi_{s}}(\widehat{Sub}_{1,L_{s},L_{s}}(\mathbb{C}^{I}\otimes
\mathbb{C}^{J}\otimes \mathbb{C}^{K})),\nonumber
\end{align}
and according to Remark \ref{nor terr}, it is equivalent to that
\begin{align}\label{norm 1ll C}
&\hat{T}^{\bot}_{\varphi_{1}}(\widehat{\mbox{\it Sub}}_{1,L_{1},L_{1}}(\mathbb{C}^{I}\otimes
\mathbb{C}^{J}\otimes \mathbb{C}^{K}))\cap
\hat{T}^{\bot}_{\varphi_{2}}(\widehat{\mbox{\it Sub}}_{1,L_{2},L_{2}}(\mathbb{C}^{I}\otimes
\mathbb{C}^{J}\otimes \mathbb{C}^{K}))\\&\subset \hat{T}^{\bot}_{\psi_{s}}(\widehat{\mbox{\it Sub}}_{1,L_{s},L_{s}}(\mathbb{C}^{I}\otimes
\mathbb{C}^{J}\otimes \mathbb{C}^{K})).\nonumber
\end{align}
Hence we need to prove that these inclusions imply that $\psi_{s}\in \{\varphi_{1},\varphi_{2}\}$.

Express a  $c'\in \{c_{1}',\ldots,c_{L_{r}}'\}$ in \eqref{psi D} as
\begin{align*}
c'=\sum_{1\leq h\leq l_{c}}
z_{0,h}c_{0,h}+\sum_{1\leq h\leq L_{1}-l_{c}}z_{1,h}c_{1,h}+\sum_{1\leq h\leq L_{2}-l_{c}}z_{2,h}c_{2,h},
\end{align*}
and write $a'=x_{1}a_{1}+x_{2}a_{2}$. We treat first the case that both $x_{1},x_{2}$ are
nonzero. From \eqref{tan psi C}, a general hyperplane in \eqref{norm 2} is a linear combination of $a^{*}_{1}\otimes
b^{*}_{2,j}\otimes c^{*}_{2,k},a^{*}_{2}\otimes
b^{*}_{1,j}\otimes c^{*}_{1,k}, a^{*}_{1}\otimes
(b^{*}_{2,j}\otimes c^{*}_{2,j}-
b^{*}_{2,k}\otimes c^{*}_{2,k}),a^{*}_{2}\otimes
(b^{*}_{1,j}\otimes c^{*}_{1,j}-
b^{*}_{1,k}\otimes c^{*}_{1,k})$.

Note that for
\begin{align*}
(x_{1}a_{1}+x_{2}a_{2})\otimes
b_{1,j}\otimes (\sum_{1\leq h\leq l_{c}}
z_{0,h}c_{0,h}+\sum_{1\leq h\leq L_{1}-l_{c}}z_{1,h}c_{1,h}+\sum_{1\leq h\leq L_{2}-l_{c}}z_{2,h}c_{2,h})
\end{align*}
in $A'\otimes B\otimes C'$, we have
\begin{align*}
H_{1}=&a^{*}_{2}\otimes b^{*}_{1,j}\otimes c^{*}_{1,k}\in \hat{T}^{\bot}_{\psi_{s}}(\widehat{Sub}_{1,L_{s},L_{s}}(\mathbb{C}^{I}\otimes
\mathbb{C}^{J}\otimes \mathbb{C}^{K}))
\end{align*}
(i.e. $H_{1}$ is tangent to $\widehat{\mbox{\it Sub}}_{1,L_{s},L_{s}}(\mathbb{C}^{I}\otimes
\mathbb{C}^{J}\otimes \mathbb{C}^{K})$ at $\psi_{s}$)
only if
\begin{align*}
&a^{*}_{2}\otimes b^{*}_{1,j}\otimes c^{*}_{1,k}\vdash \\&(x_{1}a_{1}+x_{2}a_{2})\otimes
b_{1,j}\otimes (\sum_{1\leq h\leq l_{c}}
z_{0,h}c_{0,h}+\sum_{1\leq h\leq L_{1}-l_{c}}z_{1,h}c_{1,h}+\sum_{1\leq h\leq L_{2}-l_{c}}z_{2,h}c_{2,h})\\&=x_{2}z_{1,k}=0,
\end{align*}
and respectively
\begin{align*}
H_{2}=a^{*}_{1}\otimes b^{*}_{2,j}\otimes
c^{*}_{2,k}\in \hat{T}^{\bot}_{\psi_{s}}(\widehat{\mbox{\it Sub}}_{1,L_{s},L_{s}}(\mathbb{C}^{I}\otimes
\mathbb{C}^{J}\otimes \mathbb{C}^{K})
\end{align*}
only if
\begin{align*}
&a^{*}_{1}\otimes b^{*}_{2,j}\otimes c^{*}_{2,k}\vdash \\&(x_{1}a_{1}+x_{2}a_{2})\otimes
b_{2,j}\otimes (\sum_{1\leq h\leq l_{c}}
z_{0,h}c_{0,h}+\sum_{1\leq h\leq L_{1}-l_{c}}z_{1,h}c_{1,h}+\sum_{1\leq h\leq L_{2}-l_{c}}z_{2,h}c_{2,h})\\&=x_{1}z_{2,k}=0.
\end{align*}

Therefore
$z_{1,h}=0$, for $1\leq h\leq L_{1}-l_{c}$, and $z_{2,h}=0$, for $1\leq h\leq L_{2}-l_{c}$.
So $c'=\sum_{1\leq h\leq l_{c}}z_{0,h}c_{0,h}\in C_{0}$.
For the same reason $b'\in B_{0}$ for every $b'\in \{b_{1}',\ldots,b_{L_{s}}'\}$. Thus $\psi_{s}\in A_{p}\otimes B_{0}\otimes C_{0},\ p=1,2$. But since $\dim B_{0},\ \dim C_{0}$ are both less than $\textrm{min}\ \{L_{1},\ L_{2}\}$, we get a contradiction to the fact that $\psi_{s}$ has multilinear rank $(1,L_{1},L_{1})$ or $(1,L_{2},L_{2})$. Thus $x_{1},\ x_{2}$ cannot be both nonzero.

Next, without loss of generality, we consider the case when $a'=a_{1}$. Taking
the hyperplane
\begin{align*}
H_{2}=a^{*}_{1}\otimes b^{*}_{2,j}\otimes
c^{*}_{2,k},
\end{align*}
in \eqref{norm 2}, we obtain by a similar computation that  $z_{2,h}=0$, for each $1\leq h\leq L_{2}-l_{c}$, so
\begin{align*}
c'=\sum_{1\leq h\leq l_{c}}
z_{0,h}c_{0,h}+\sum_{1\leq h\leq L_{1}-l_{c}}z_{1,h}c_{1,h}
\in C_{1}\oplus C_{0};
\end{align*}
for the same reason, $b_{h}'\in B_{1}\oplus B_{0}$, and
\begin{align*}
\psi_{s}\in A_{1}\otimes
(B_{1}\oplus B_{0})\otimes (C_{1}\oplus C_{0}).
\end{align*}
Now \eqref{phi1+phi2}
is valid if and only if
\begin{align*}
[\psi_{s}(A^{*}_{1})]=[\varphi_{1}(A_{1}^{*})],
\end{align*}
and then
$[\psi_{s}]=[\varphi_{1}]$. Thus we obtain Case D.

\subsection{Proof that Condition E implies generic uniqueness}
The proof will be nearly the same as that for Condition D. As there, it is sufficient to prove the case $I=R$, $K=\sum^{R}_{r=1}L_{r}$. Let $A,\ B$ and $C$ denote vector spaces of dimensions $I,\ J,\ K$ respectively. Choose splitting $A=\bigoplus_{1\leq p\leq R}A_{p}$ and $C=\bigoplus_{1\leq r\leq R}C_{r}$. Further fix a basis $\{b_{1},\ldots,b_{J}\}$ for $B$.

Choose general points $\varphi_{p}\in \widehat{\mbox{\it Sub}}_{1,L_{p},L_{p}}(\mathbb{C}^I\otimes
\mathbb{C}^J\otimes \mathbb{C}^{K})$ for $1\leq p\leq R$. Without loss of generality, for $1\leq p\leq R$, we can assume
\begin{align*}
\varphi_{p}=a_{p}\otimes (b_{p,1}\otimes
c_{p,1}+b_{p,2}\otimes c_{p,2}+\cdots+b_{p,L_{p}}\otimes
c_{p,L_{p}})\in A_{p}\otimes B_{p}\otimes C_{p},
\end{align*}
where $\{a_{p}\}$, $\{b_{p,1},\ldots,b_{p,L_{p}}\}\subset \{b_{1},\ldots,b_{J}\}$,
$\{c_{p,1},\ldots,c_{p,L_{p}}\}$ are bases for $A_{p}$, $B_{p}$, $C_{p}$, respectively.

Then for $\varphi_{p}\in A_{p}\otimes B_{p}\otimes C_{p},1\leq p\leq R$, it is clear that
\begin{align*}
&\hat{T}^{\bot}_{\varphi_{p}}(\widehat{\mbox{\it Sub}}_{1,L_{p},L_{p}}(\mathbb{C}^{I}\otimes
\mathbb{C}^{J}\otimes
\mathbb{C}^{K}))=(A^{\bot}_{p}\otimes
B^{\bot}_{p}\otimes C^{*})\oplus (A_{p}^{*}\otimes
B_{p}^{\bot}\otimes C^{\bot}_{p})\\&\oplus (A_{p}^{\bot}\otimes
B_{p}^{*}\otimes C^{\bot}_{p})\oplus (A_{p}^{\bot}\otimes
(\varphi_{p}(A^{*}_{p})^{\bot}\cap (B^{*}_{p}\otimes
C^{*}_{p}))).
\end{align*}
Due to Theorem \ref{ter lem} and since $J>2L_{R}\geq L_{i}+L_{r}$, there exists $b^{*}_{q,j}\notin  B^{*}_{i}\oplus B^{*}_{r}$ for any $i,r$, where $i\neq r$, such that
\begin{align}
&\hat{T}^{\bot}_{\sum
\varphi_{p}}(\textbf{J}(\widehat{\mbox{\it Sub}}_{1,L_{1},L_{1}}(\mathbb{C}^{I}\otimes
\mathbb{C}^{J}\otimes
\mathbb{C}^{K}),\ldots,\widehat{\mbox{\it Sub}}_{1,L_{R},L_{R}}(\mathbb{C}^{I}\otimes
\mathbb{C}^{J}\otimes
\mathbb{C}^{K})))\nonumber\\ &\supset\bigoplus_{i\neq r}
\langle a^{*}_{i}\otimes b^{*}_{q,j}\otimes c^{*}_{r,k}\rangle,1\leq j\leq L_{q},1\leq k\leq L_{r}.
\end{align}
Let
\begin{align}\label{psi E}
\psi_{s}=a'\otimes (b_{1}'\otimes c_{1}'+\cdots+b_{L_{s}}'\otimes
c_{L_{s}}')\in A'\otimes B'\otimes C',
\end{align}
be a general point in one of $\widehat{\mbox{\it Sub}}_{1,L_{s},L_{s}}\mathbb{C}^{I}\otimes \mathbb{C}^{J}\otimes
\mathbb{C}^{K},\ 1\leq s\leq R$, where $A'=\langle a' \rangle$, $B'=\langle b_{1}',\ldots,b_{L_{s}}'\rangle$ and $C'=\langle c_{1}',\ldots,c_{L_{s}}'\rangle$, and note that
\begin{align}
&\hat{T}_{\psi_{s}}(\widehat{\mbox{\it Sub}}_{1,L_{s},L_{s}}\mathbb{C}^{I}\otimes \mathbb{C}^{J}\otimes
\mathbb{C}^{K})\\&=(A'\otimes B\otimes C')+
(A'\otimes B'\otimes C)+(A\otimes \psi(A'^{*})).\nonumber
\end{align}
If
\begin{align}\label{phi1+phi R}
[\varphi_{1}+\cdots+\varphi_{R}]=[\psi_{1}+\cdots+\psi_{R}],
\end{align}
we have
\begin{align}\label{tan 1ll E}
&\hat{T}_{\varphi_{1}}(\widehat{\mbox{\it Sub}}_{1,L_{1},L_{1}}(\mathbb{C}^{I}\otimes
\mathbb{C}^{J}\otimes \mathbb{C}^{K}))+\cdots +
\hat{T}_{\varphi_{R}}(\widehat{\mbox{\it Sub}}_{1,L_{R},L_{R}}(\mathbb{C}^{I}\otimes
\mathbb{C}^{J}\otimes \mathbb{C}^{K}))\\&\supset \hat{T}_{\psi_{s}}(\widehat{\mbox{\it Sub}}_{1,L_{s},L_{s}}(\mathbb{C}^{I}\otimes
\mathbb{C}^{J}\otimes \mathbb{C}^{K})),\nonumber
\end{align}
and according to Remark \ref{nor terr}, it is equivalent to that
\begin{align*}
\bigcap_{1\leq p\leq R}\widehat{T}^{\bot}_{\varphi_{p}}(\widehat{\mbox{\it Sub}}_{1,L_{p},L_{p}}(\mathbb{C}^{I}\otimes
\mathbb{C}^{J}\otimes
\mathbb{C}^{K}))\subset \hat{T}^{\bot}_{\psi_{s}}(\widehat{\mbox{\it Sub}}_{1,L_{s},L_{s}}(\mathbb{C}^{I}\otimes
\mathbb{C}^{J}\otimes \mathbb{C}^{K})).
\end{align*}
Hence we need to prove that these inclusions imply that $\psi_{s}\in \{\varphi_{1},\cdots,\varphi_{R}\}$.

Express $a'$ and $c'$  in one of the $\{c_{i}': 1\leq i\leq L_{s}\}$ occurring in \eqref{psi E} as
\begin{align*}
&a'=x_{1}a_{1}+ \cdots +x_{R}a_{R},\\ &c'=\sum_{1\leq h\leq L_{1}}z_{1,h}c_{1,h}+\cdots+\sum_{1\leq h\leq L_{R}}z_{R,h}c_{R,h}.
\end{align*}
Since $J> 2L_{R},\ {{J}\choose{L_{R}}}\geq R$, without loss of generality, we first treat the case when both
 $x_{1}, x_{2}$ are nonzero.

A general hyperplane in (4.15) is a linear combination of $a^{*}_{i}\otimes b^{*}_{q,j}\otimes c^{*}_{r,k}$, and for $b^{*}_{q,j}\notin B^{*}_{i}\oplus B^{*}_{r}$, we have
\begin{align*}
H=a^{*}_{i}\otimes b^{*}_{q,j}\otimes c^{*}_{r,k}\in \hat{T}^{\bot}_{\psi_{s}}(\widehat{\mbox{\it Sub}}_{1,L_{s},L_{s}}(\mathbb{C}^{I}\otimes
\mathbb{C}^{J}\otimes \mathbb{C}^{K})),
\end{align*}
only if
\begin{align*}
&a^{*}_{i}\otimes b^{*}_{p,j}\otimes c^{*}_{r,k}\vdash \\&
(x_{1}a_{1}+ \cdots +x_{R}a_{R})\otimes b_{p,j}\otimes (\sum_{1\leq h\leq L_{1}}z_{1,h}c_{1,h}+\cdots+\sum_{1\leq h\leq L_{R}}z_{R,h}c_{R,h})\\&=x_{i}z_{t,h}=0,\ t\neq i,\ \textrm{for}\ i=1,2,
\end{align*}
where
\begin{align*}
&(x_{1}a_{1}+ \cdots +x_{R}a_{R})\otimes b_{p,j}\otimes (\sum_{1\leq h\leq L_{1}}z_{1,h}c_{1,h}+\cdots+\sum_{1\leq h\leq L_{R}}z_{R,h}c_{R,h})\\& \in A'\otimes B\otimes C'.
\end{align*}
Therefore $z_{t,h}=0$, when $t\neq 1,2$. Then $c'_{1}=\cdots=c'_{L_{s}}=0,\ b'_{1}=\cdots=b'_{L_{s}}=0$,
hence we have $\psi_{s}=0$ in this case.

Next without loss of generality, we treat the case
 $a'=a_{1}$. Taking the hyperplane $H=a^{*}_{1}\otimes
b^{*}_{q,j}\otimes c^{*}_{r,k}$, we obtain by similar computation that $z_{t,h}=0$, when
$t\neq 1$. So
$c'=\sum_{1\leq h\leq L_{1}}z_{1,h}c_{1,h}\in C_{1}$.
For the same reason, we have
$b'\in B_{1}$, for every $b'\in \{b_{1}',\ldots,b_{L_{s}}'\}$ and consequently $\psi_{s}\in A_{1}\otimes
B_{1}\otimes C_{1}$.

Now
\eqref{phi1+phi R} is valid only if
\begin{align*}
[\psi_{s}(A^{*}_{1})]=[\varphi_{1}(A_{1}^{*})]
\end{align*}
and then $[\psi_{s}]=[\varphi_{1}]$.
Generally when $a'=a_{q}$, we have
$[\psi_{s}(A^{*}_{q})]=[\varphi_{q}(A_{q}^{*})]$, which implies
$[\psi_{s}]=[\varphi_{q}]$. Thus we obtain Case E.

\section{A criterion of uniqueness}
We will give a new proof of a criterion of uniqueness for block term tensor decomposition, due to De Lathauwer \cite{Lat 4}. Note that the uniqueness condition in our Corollary \ref{twd unique} concerns a larger class of decompositions.
\begin{thm}(\cite{Lat 4}, Theorem 2.3)\label{de lautherwal thm}
Assume $I\geq R$, \eqref{BTD1LL} is essentially unique if and only if for any $X_{j_{1}},\ \cdots,\ X_{j_{s}}$, we have
\begin{align}\label{geometric lat}
\langle X_{j_{1}},\ \cdots,\ X_{j_{s}}\rangle\cap \sigma_{L_{j_{t}}}(\mathbb{P}^{J-1}\times \mathbb{P}^{K-1})\subset \{X_{j_{1}},\ \cdots,\ X_{j_{s}}\},\ 1\leq t\leq s.
\end{align}
\end{thm}
\begin{proof}
Assume the contrary that $Y=\sum_{r=1}^{R}a'_{r}\otimes \tilde{X}_{r}$ is different from \eqref{BTD1LL}. Since $a_{1},\ldots,a_{R}$ are independent, we have $a'_{r}=\sum^{R}_{j=1}\alpha^{r}_{j}a_{j}$, where $\alpha^{r}_{j}$ are not all zero.
Hence
\begin{align*}
Y=\sum_{r=1}^{R}a_{r}\otimes X_{r}=\sum_{r=1}^{R}a_{r}\otimes (\sum^{R}_{j=1} \alpha^{r}_{j} \tilde{X}_{j}).
\end{align*}
Thus $X_{r}=\sum^{R}_{j=1} \alpha^{r}_{j} \tilde{X}_{j}$. Taking the inverse of $[\alpha^{r}_{j}]$, we have $\tilde{X}_{r}=\sum^{R}_{j=1} \tilde{\alpha}^{r}_{j} X_{j}$.
Consequently, there exists $r,j_{1},j_{2}\in \{1,\ldots,R\}$ such that $j_{1}\neq j_{1}$ and $\tilde{\alpha}^{r}_{j_{1}}\cdot\tilde{\alpha}^{r}_{j_{2}}\neq 0$. Then we obtain
\begin{align*}
\tilde{X}_{r}\in\langle X_{j_{1}}, \ldots, X_{j_{s}}\rangle\cap \sigma_{L_{r}}(\mathbb{P}^{J-1}\times \mathbb{P}^{K-1}).
\end{align*}
But $\tilde{X}_{r}$ is not belong to $\{X_{j_{1}}, \ldots, X_{j_{s}}\}$, which contradicts to \eqref{geometric lat}.
\end{proof}

\section{Acknowledgement}

I am very grateful to Prof. Landsberg for many inspiring
discussions, and unfailingly useful suggestions. Without his
guidance in respect of the geometric intuition, this article will
never be successfully completed. I am also very grateful to Prof.
Ottaniani to teach us the concept of tangential weak-defectivity and
to Prof. Oeding for very useful comments. By
email conversation, Prof. De Lathauwer  shared with me his
unpublished results on block component analysis and Prof. Comon
gave me a lot of reference on blind source separation. Words can
not express my gratitude to them.
\bibliographystyle{amsplain}

% ----------------------------------------------------------------
%\bibliographystyle{amsplain}

\end{document}